\documentclass[a4paper]{amsart}
\usepackage{color}
\usepackage{amsmath,amssymb,amsfonts,amsthm,bm}
\usepackage{mathrsfs}
\usepackage{graphicx}
\usepackage{enumerate}
\usepackage[numbers, sort&compress]{natbib}
\usepackage[shortlabels]{enumitem}

\newtheorem{theorem}{Theorem}[section]
\newtheorem{corollary}[theorem]{Corollary}
\newtheorem{lemma}[theorem]{Lemma}
\newtheorem{proposition}[theorem]{Proposition}
\theoremstyle{definition}

\theoremstyle{remark}
\newtheorem{remark}[theorem]{Remark}
\numberwithin{equation}{section}

\newcommand\RR{{{\mathbb R}}}

\makeatletter

\@namedef{subjclassname@2020}{%
\textup{2020} Mathematics Subject Classification}

\makeatother

\begin{document}

\title[The non cutoff Kac equation]
{Sharp regularizing effect of the Cauchy problem
for inhomogeneous non-cutoff  Kac equation}
\author[X. Cai, H. Cao \& C.-J. Xu] {Xinzhi Cai, Hongmei Cao and Chao-Jiang Xu}
\address{School of Mathematics and Key Laboratory of Mathematics  MIIT,
Nanjing University of Aeronautics and Astronautics, Nanjing 210016, China}
\email{Math22cxz@nuaa.edu.cn; hmcao\_91@nuaa.edu.cn; xuchaojiang@nuaa.edu.cn}

\date{}

\begin{abstract}
In this work, we study the spatially inhomogeneous Kac equation with a non-cutoff cross section in a setting close to equilibrium. We prove that the solution to the Cauchy problem exhibits a sharp Gevrey-Gelfand-Shilov smoothing effect with an optimal radius. We employ a well-chosen exponential-type Fourier multiplier to establish the smoothing effect for position and velocity variables. 
\end{abstract}

\maketitle
\section{Introduction}
\label{section1}

The Cauchy problem for spatially inhomogeneous nonlinear Kac's equation is written as:
\begin{equation}
\left\{
\begin{array}{ll}
\partial_t f+v\,\partial_x f=K(f,\,\,f), &(x,  v)\in \mathbb{R}^2,\,\,t>0, \\
f|_{t=0}=f_{0}(x, v),
\end{array}
\right.  \label{1.1}
\end{equation}%
where $f=f(t,\,x,\, v)$ is the nonnegative density distribution function of
particles at position $x\in\RR$ with velocity $v\in {{\mathbb{R}}}$ at time $t$. The right hand
side of equation (\ref{1.1}) is given by Kac's bilinear collisional operator with respect to the velocity variable,
\begin{equation*}
K(f,\,\,g)=\int_{{{\mathbb{R}}}}\int_{-\pi /2}^{\pi /2}\beta (\theta
)\left\{ f(v_{\ast }^{\prime })g(v^{\prime })-f(v_{\ast })g(v)\right\}
d\theta dv_{\ast }\,,
\end{equation*}%
where the pre-post collision velocities can be defined by
\begin{equation*}
v^{\prime }=v\,\cos \theta -v_{\ast }\,\sin \theta , \ \ \  v_{\ast
}^{\prime }=v\,\sin \theta +v_{\ast }\,\cos \theta .
\end{equation*}%
We suppose that the cross-section kernel is non cut-off in the following sense:
\begin{equation}
\beta (\theta )=C_{0}\frac{|\cos \theta |\,\,\,\,\,}{\,\,|\sin \theta
|^{1+2s}},\,\,\,\,\,-\frac{\pi }{2}\leq \theta \leq \frac{\pi }{2}\,,
\label{1.1+1}
\end{equation}%
where $0<s<1$ and $C_{0}>0$.

In a situation close to the equilibrium framework, consider the fluctuation of density distribution function
$$
f(t, x, v)=\mu(v)+\sqrt{\mu}\,g(t, x, v), \ \ \mu(v)=(2\pi)^{-\frac{1}{2}}e^{-\frac{v^2}{2}}.
$$
Note that $K(\mu,\mu)=0$ and, we turn to the following Cauchy problem
\begin{equation} \label{eq-1}
\left\{
\begin{aligned}
&\partial_t g+v\partial_xg+\mathcal{L}g=\mathcal{K}(g, g),\,\,\,\, (x, v)\in\mathbb{R}^2,
\ \   t>0,\\
&g|_{t=0}=g_{0}(x, v),
\end{aligned} \right.
\end{equation}
where
\begin{equation}\label{KKK}
    \mathcal{K}(f, g)=\mu^{-\frac{1}{2}}K(\mu^{\frac{1}{2}}f, \mu^{\frac{1}{2}}g),
\end{equation}
and
$$
\mathcal{L} g=-\mathcal{K}(\mu^{1/2},  g) -\mathcal{K}(g, \mu^{1/2} )= \mathcal{L}_1 g +\mathcal{L}_2 g,
$$

The kernel of linear operator $\mathcal{L}$ is given by
$$
\mathcal{N}=\mathrm{Span}\{\sqrt{\mu},\, v\sqrt{\mu},\, |v|^2\sqrt\mu \}.
$$
Kac's equation is the simplified one-dimensional model of  Boltzmann equation, see \cite{nine}. There exists a series of works about the existence and regularity of solution for non-cutoff inhomogeneous Boltzmann equations (see \cite{two,six,norm,ten,D95,mildsolution,duan,expv2,hardpotential}).

We introduce now some notations, the following function spaces will be used. For $1\leq p\leq+\infty , \ell \in {{\mathbb{R}}}$,
\begin{equation*}
L_{\ell }^{p}({{\mathbb{R}}})=\Big\{f; \,\,\Vert f\Vert _{L_{\ell }^{p}}=\Big(%
\int_{{{\mathbb{R}}}}|\langle v\rangle ^{\ell }f(v)|^{p}dv\Big)%
^{1/p}<+\infty \Big\},
\end{equation*}
where $\langle \ \cdot\ \rangle =(1+|\ \cdot\ |^{2})^{\frac 12}$. For $r \in {{\mathbb{R}}}$,
$H^{r}(\mathbb{R})$ is the classical Sobolev space,  
and $ H^r_x(L^2_v)=H^r(\mathbb{R}_x;L^2(\mathbb{R}_v))$ .

Given $\sigma \geq 1$, the Gevrey class $\mathcal{G}^{\sigma} (\mathbb{R}^n )$ (of index $\sigma$) is defined as the set of all functions $f \in C ^ {\infty} ( \mathbb{R}^n)$ such that there exists a $C> 0$ satisfying
\begin{equation*}
\| \partial ^ { \alpha } f\|_{L^2(\mathbb R^n)} \leq C ^ { |\alpha| + 1 } ( \alpha ! ) ^ {\sigma }, \qquad \forall\,\alpha \in \mathbb{N}^n,
\end{equation*}
or equivalently, there exists $c_0>0$ such that $e^{c_0\langle D\rangle^{\frac 1\sigma}} f\in L^2(\mathbb{R}^n)$, where $c_0$ is called the Gevrey radius.  The analytic space is with $\sigma=1$ and we denote by  $\mathcal{G}^1(\mathbb R^n)=\mathcal{A}(\mathbb{R}^n)$. 

The Gelfand-Shilov space $S^\mu_\nu(\mathbb{R}^d)$ with $\mu, \nu>0, \mu+\nu\ge 1$ is
$$
\|x^\beta\partial^\alpha{f}\|_{L^2(\mathbb R^n)} \leq C^{|\alpha|+|\beta|+1}(|\alpha|!)^{\mu}
(|\beta|!)^{\nu},\quad  \forall\alpha, \beta\in\mathbb{N}^n,
$$
it is equivalent to  
\begin{equation*}\label{1-1-2bbb}
e^{c_0(\langle D_x\rangle^{\frac1{\mu}}+\langle x\rangle^{\frac1{\nu}})}f\in L^2(\mathbb{R}^n),
\end{equation*}
where $c_0$ is called the Gelfand-Shilov radius.

For the spatial inhomogeneous non cut-off Kac equation, the paper \cite{LMPX}
has proven the existence of the solution and the smoothing effect of Cauchy problem \eqref{eq-1}, furthermore
$$
g\in L^\infty(]0, T[; \mathcal{G}^{1+\frac{1}{2s}}(\mathbb{R}_x; S^{1+\frac{1}{2s}}_{1+\frac{1}{2s}}(\mathbb{R}_v))),
$$ 
for any $T>0$.  \cite{three} study the well-posedness when initial datum belonging to the spatially critical Besov space, and improves the  Gevrey index of  \cite{LMPX}. The purpose of this paper is to improve the regularity results in \cite{three} to the optimal index which is mentioned in \cite{hardpotential} for Boltzmann equation, but we obtain here the optimal Gevrey-Gelfand-Shilov radius.

The following is the main Theorem of this work.

\begin{theorem}\label{theo1}
For $r>\frac{1}{2}$, assume that the initial datum $g_{0}\in H^r_x(L^2_v)$, and the cross-section
$\beta $ satisfy \eqref{1.1+1} with $0<s<1$. For all $g_0$ satisfies
$$
||g_0||_{H^r_x(L^2_v)}\leq\epsilon_0
$$
with $\epsilon_0>0$ a small constant,
the Cauchy problem \eqref{eq-1} admits a unique global solution such that, for any $T>0 $
\begin{align*}
g&\in L^\infty(]0, T[;  \mathcal{G}^{\frac{1}{{2s}}}(\mathbb{R}_{x};\, S^{\frac 1{2s}}_{\frac 1{2s}}(\mathbb R_v))), \quad \text{when} \quad 0<s< 1/2,
\\
g &\in L^\infty(]0, T[; \mathcal{A}(\mathbb{R}_{x};\, S^{1}_{1}(\mathbb R_v))), \qquad \text{when} \quad 1/2\leq s< 1. 
\end{align*}
\end{theorem}

\begin{remark}
In fact, we have proven
\begin{equation*}
e^{c_0(t^{1+2\widetilde{s}}(1-\Delta_x)^{\widetilde{s}}+t(1-\Delta_v+|v|^2)^{\widetilde{s}})}g(t)\in L^\infty([0,\, T]; \,H^r_x(L^2_v)),
\end{equation*}
with $\widetilde{s}=\min\{s, \frac 1 2\}$. So that the $2\tilde{s}$-Gevrey radius for $x$ variable is $c_0t^{1+2\widetilde{s}}$, and 
$2\tilde{s}$-Gelfand-Shilov radius for $v$ variable is $c_0t$. Compare with the results of Boltzmann equation in \cite{hardpotential}, where they prove  the $2\tilde{s}$-Gevrey smoothing effect of $(x, v)$ variables, and the Gevrey radius is $c_0t^a$ with a big $a>1$, so we improve considerably the Gevrey radius, and we gain also the exponential decay for velocity variable. 
\end{remark}

\begin{remark}
The results of this Theorem improves the results in \cite{LMPX,LX,three,G-N2013,paperMorimotoUkai}. Our computation was inspired by the work in \cite{four} for the Landau equation( other improvement can be seen in \cite{LAndau,one} ), but our problem is like a non linear form of pseudo-differential operators, so that the computation is more complicated. The index and radius of Gevery properties is claimed to be sharp, see \cite{four,hardpotential}.
\end{remark}

\vskip0.5cm
\section{Schema for the proof of Main theorem }\label{section2}\smallskip

The paper \cite{LMPX,three} has already proven that the Cauchy problem \eqref{eq-1} admet a unique solution   
\begin{equation}\label{2.10}
g\in L^\infty(]0, T[; \mathcal{G}^{1+\frac{1}{2s}}(\mathbb{R}_x; S^{1+\frac{1}{2s}}_{1+\frac{1}{2s}}(\mathbb{R}_v))),
\end{equation}
and
\begin{equation}\label{2.11}
    \|g\|_{L^\infty([0, T]; H^r_x(L^2_v))}\le C_0\epsilon_0.
\end{equation}
and it is the limit of the following approximation system:
\begin{equation}\label{2.11b}
\begin{cases}
\partial_t g^{n+1}+v\,\partial_x g^{n+1}+\mathcal{L} g^{n+1}=\mathcal{K}(g^{n}, g^{n+1}),\quad n\ge 0\\
g^{n+1}|_{t=0}=g_{0}(x, v),\\
g^0=e^{-t(1-\triangle_{x, v})}g_0.
\end{cases}
\end{equation}
The existence, regularity, smallness and the convergence of $\{ g^n\}$ with norm as \eqref{2.10} and \eqref{2.11} is given in  \cite{LMPX,three}.  

We just need to improve the Gevrey index of above approximation sequence, and try to prove
\begin{equation}\label{2.1}
\|e^{c_0(t^{1+2\widetilde{s}}(1-\Delta_x)^{\widetilde{s}}+t(1-\Delta_v)^{\widetilde{s}})}g^{n}\|_{L^\infty([0,\, T]; \,H^r_x(L^2_v))}\le C_0\epsilon_0,
\end{equation}
with $\widetilde{s}=\min\{s, \frac 12\}$ and $C_0$ independent of $n$, this imply,
$$
g\in L^\infty(]0, T[; \mathcal{G}^{\frac{1}{2{s}}}(\mathbb{R}^2_{x, v})),\ \ \ \ 
0<s<\frac 12,
$$ 
and
$$
g\in L^\infty(]0, T[; \mathcal{A}(\mathbb{R}^2_{x, v})),\ \ \ \ \frac 12\le s<1.
$$ 
To understand the idea of proof of \eqref{2.1}, we consider the following  Cauchy problem of the fractional Kolmogorov equation
\begin{equation}\label{Kolmogorov}
\left\{
\begin{array}{ll}
\partial_t f+v\,\partial_x f+(1-\Delta_v)^s f=0, &(x,  v)\in \mathbb{R}^{2n},\,\,t>0, \\
f|_{t=0}=f_{0}(x, v)\in L^2(\RR^{2n}),
\end{array}
\right.  
\end{equation}%
by using Fourier transform, we obtain 
$$
\widehat{f}(t,\eta,\xi)=e^{-\int_0^t\langle \xi+\rho\eta\rangle^{2s}d\rho}\widehat{f}_0(\eta,\xi+t\eta).
$$
Use now the inequality (see \cite{four} ), for $\alpha>0$,
\begin{equation}\label{Ukai}
\int^t_0 (1+| \xi +\rho \eta |^2) ^{\alpha/2} d\rho \ \sim \  t\left \{1+ |\xi|^\alpha+ t^\alpha |\eta|^\alpha\right \},
\end{equation}
we can obtain
\begin{align*}
    e^{c(t(1-\Delta_v)^s+t^{1+2s}(1-\Delta_x)^s)}f\in L^2(\RR_{x,v}^{2n}).
\end{align*}
Here the equation in \eqref{Kolmogorov} is linear with constent coefficients, if we come back to the non linear equation \eqref{eq-1}, we have to use the covexity of exponential function which need
$$
\langle \xi\rangle^{\sigma}\le \langle \xi-\eta\rangle^{\sigma}+\langle \eta\rangle^{\sigma},\quad \forall \xi,\ \eta\in\mathbb R^d, 
$$
and this hold true only for $\sigma\le 1$, 
so that we have to replace $2s$ by $2\tilde{s}\le 1$, then we try to prove
\begin{equation}\label{2.1b}
\|e^{\Psi_s(t, D_x, D_v)}g^{n}\|_{L^\infty([0,\, T]; \,{H^r_x( L^2_v)})}\le C_0\epsilon_0,
\end{equation}
with
$$
\Psi_s(t, \eta, \xi)= c_0 \int^t_0\langle \xi+\rho\eta\rangle^{2\widetilde{s}} d\rho,
$$
But our solution of \eqref{2.11b} satisfy only \eqref{2.10}, we can't use 
$$
e^{2\Psi_s(t, D_x, D_v)}g^{n+1}(t)
$$
as test function to the equation in \eqref{2.11b}, then we will choose
\begin{equation}\label{DM}
    M_{\delta}(t,\eta,\xi)=\frac{e^{\Psi_s(t,\eta,\xi)}}{1+\delta e^{\Psi_s(t,\eta,\xi)}}, \ \ \ 0<\delta\ll 1,
\end{equation}
so that we have, for any $0<\delta\ll 1$ 
$$
M^2_{\delta}(t,D_x, D_v) g^{n+1}(t)\in L^\infty(]0, T[; \mathcal{G}^{1+\frac{1}{2s}}(\mathbb{R}_x; S^{1+\frac{1}{2s}}_{1+\frac{1}{2s}}(\mathbb{R}_v))),
$$
it is a good test function for the equation in \eqref{2.11b}, since $M^2_{\delta}(t,D_x, D_v)$ is a pseudo-differential operators of order $0$, then we have, for any $n\in\mathbb{N}$, 
\begin{equation}\label{2.6a}
    \begin{aligned}
        \Big((\partial_t +v\partial_x)g^{n+1},\ & M_{\delta}^2g^{n+1} \Big)_{H^r_x(L^2_v)}+\Big(\mathcal{L}g^{n+1},\ M_{\delta}^2g^{n+1} \Big)_{H^r_x(L^2_v)}\\
        &\qquad \qquad =\Big(\mathcal{K}(g^{n},\ g^{n+1}),M_{\delta}^2g^{n+1} \Big)_{H^r_x(L^2_v)}.
    \end{aligned}
\end{equation}
Our object is, by using the above equation, to establish the following estimate
\begin{equation}\label{2.7a}
    \|M_{\delta} g^{n}(t) \|_{H^r_x(L^2_v)}\le C\epsilon_0,\ \ \ \ \ 
    0\le t\le T,\ \ \ n\in\mathbb{N},
\end{equation}
with $C$ independs of $0<\delta\ll 1$ and $n$, then take $\delta\,\to\, 0$ in \eqref{2.7a} to get \eqref{2.1b}. 

We will prove \eqref{2.7a} by induction on $n\in\mathbb{N}$, it is true for $n=0$ by the choose of $g^0=e^{-t(1-\triangle_{x, v})}g_0$ and hypothesis for $g_0$ in Theorem \ref{theo1}. 

We study the three terms in \eqref{2.6a} by the following three Propositions,
\begin{proposition} \label{proposition2.1}
For the kinetic part of  \eqref{2.6a}, then for any function $g$ satisfy \eqref{2.10}, we have 
\begin{equation}\label{2.9}
    \Big((\partial_t +v\partial_x)g, M_{\delta}^2g \Big)_{H^r_x(L^2_v)}\ge \frac{1}{2}\frac{d}{dt}||M_{\delta}g||^2_{H^r_x(L^2_v)}-c_0||\langle D_v\rangle^sM_{\delta}g||^2_{H^r_x(L^2_v)}.  
\end{equation}
\end{proposition}
\begin{proof} By using Plancherel formula, we have   
\begin{align*}
\Big((\partial_t +v\partial_x)g, M_{\delta}^2g \Big)_{H^r_x(L^2_v)}&=\Big((\partial_t -\eta\partial_\xi)\hat{g},\  \langle \eta\rangle^{2 r}M_{\delta}^2(t, \eta, \xi)\hat{g} \Big)_{L^2_{\eta, \xi}}\\
&=\Big((\partial_t -\eta\partial_\xi)\big(\langle \eta\rangle^{ r}M_{\delta}(t, \eta, \xi)\hat{g}\big),\  \langle \eta\rangle^{ r}M_{\delta}(t, \eta, \xi)\hat{g} \Big)_{L^2_{\eta, \xi}}\\
&\qquad -\Big(\{(\partial_t -\eta\partial_\xi)M_{\delta}(t, \eta, \xi)\}\langle \eta\rangle^{ r}\hat{g},\  \langle \eta\rangle^{ r}M_{\delta}(t, \eta, \xi)\hat{g} \Big)_{L^2_{\eta, \xi}},
 \end{align*}
here, for the last term, using the fact
\begin{align*}
(\partial_t -\eta\partial_\xi)&\Psi_s(t,\eta,\xi)=c_0\langle \xi\rangle^{2\tilde{s}},\\
(\partial_t -\eta\partial_\xi) M_\delta(t,\eta,\xi)&=\frac{M_\delta(t,\eta,\xi)}{1+\delta e^{\Psi_s}}(\partial_t -\eta\partial_\xi)\Psi_s(t,\eta,\xi)\\
&=\frac{M_\delta(t,\eta,\xi)}{1+\delta e^{\Psi_s}}c_0\langle \xi\rangle^{2\tilde{s}}. 
\end{align*}
We get
\begin{align*}
\Big((\partial_t +v\partial_x)g, M_{\delta}^2g \Big)_{H^r_x(L^2_v)}&=\frac 12\frac{d}{dt}\|M_\delta g\|^2_{H^r_x(L^2_v)}\\
&\quad -c_0\Big(\frac{M_{\delta}(t, \eta, \xi)}{1+\delta e^{\Psi_s}}\langle \xi\rangle^{2\tilde{s}}\langle \eta\rangle^{ r}\hat{g},\  \langle \eta\rangle^{ r}M_{\delta}(t, \eta, \xi)\hat{g} \Big)_{L^2_{\eta, \xi}},
 \end{align*}
on the other hand,
\begin{align*}
    \left|\Big(\frac{M_{\delta}(t, \eta, \xi)}{1+\delta e^{\Psi_s}}\langle \xi\rangle^{2\tilde{s}}\langle \eta\rangle^{ r}\hat{g},\  \langle \eta\rangle^{ r}M_{\delta}(t, \eta, \xi)\hat{g} \Big)_{L^2_{\eta, \xi}}\right|\le
    ||\langle D_v\rangle^sM_{\delta}g||^2_{H^r_x(L^2_v)},
\end{align*}
which imply then \eqref{2.9}.
\end{proof}

For the second term in  \eqref{2.6a}, we need to define an anisotropic norm associated with the linear operators $\mathcal{L} $, 
$$
\begin{aligned}
||| g|||^2=& \int_{\mathbb R^2_{v_\ast, v}}\int^{\frac \pi 2}_{-\frac \pi 2} \beta(\theta) \mu_* \big(g'-g\big)^2d\theta dv_\ast dv
\\&\qquad
+
\int_{\mathbb R^2_{v_\ast, v}}\int^{\frac \pi 2}_{-\frac \pi 2}\beta(\theta) g^2_* \big(\sqrt{\mu'} - \sqrt{\mu}\big)^2 d\theta dv_\ast dv,
\end{aligned}
$$
and
$$
||| g|||_{(r,0)}^2= \int_{\mathbb R} |||\langle D_x\rangle^r g|||^2 dx.
$$
The relationship of this anisotropic norm with the linear operators $\mathcal{L}$ is
\begin{equation}\label{2.2.1xb}
\Big(\mathcal{L} g,\, g\Big)_{H^r_x(L^2_v)}+C\|g\|^2_{H^r_x(L^2_v)} \sim ||| g|||^2_{(r,0)},
\end{equation}
and
\begin{equation}\label{2.2.1xc}
\|g\|^2_{H^r_x(H^s_v)}+||\langle v\rangle^s g||^2_{H^r_x(L^2_v)}\le C|||g|||^2_{(r, 0)}. 
\end{equation}

The analysis of linear operator $\mathcal{L}$ and anisotropic norm $|||\cdot|||$ can be seen in the Sction 2 of \cite{norm} or \cite{Upper1,Upper2}.

\begin{proposition}\label{proposition2.2}
For the linear part of  \eqref{2.6a}, there exists $c_1>0, C_1>0$, such that, for any function $g$ satisfy \eqref{2.10}, we have  
$$
    \big(\mathcal{L}g, M_{\delta}^2g \big)_{H^r_x(L^2_v)}\ge c_1 ||| M_{\delta} g|||_{(r,0)}^2 -C_1||M_{\delta}g||^2_{H^r_x(L^2_v)}.   
$$
\end{proposition}
Finally, for the thrid term in  \eqref{2.6a},
\begin{proposition}\label{proposition2.3}
For the nonlinear part of  \eqref{2.6a}, there exists $C_2>0$, such that, for any function $f, g, h$ satisfy \eqref{2.10}, we have  
$$
\begin{aligned}
    \left|\big(M_{\delta}\mathcal{K}(f, g), h \big)_{H^r_x(L^2_v)}\right|&\le C_2 
    \|M_{\delta} f\|_{H^r_x(L^2_v)}||| M_{\delta} g|||_{(r,0)}||| h|||_{(r,0)} . 
\end{aligned}
$$
\end{proposition}
The proof of above two Propositions is the main technical part of this work, and will give by the rest of this paper.

Combine above three Propositions, from \eqref{2.6a}, we get 
$$
    \begin{aligned}
        \frac{1}{2}\frac{d}{dt}||M_{\delta}g^{n+1}||^2_{H^r_x(L^2_v)}&+  c_1 ||| M_{\delta} g^{n+1}|||_{(r,0)}^2 \le c_0||\langle D_v\rangle^sM_{\delta}g^{n+1}||^2_{H^r_x(L^2_v)}     \\
        &+C_1||M_{\delta}g^{n+1}||^2_{H^r_x(L^2_v)} +C_2 
    \| M_{\delta} g^{n}\|_{H^r_x(L^2_v)}||| M_{\delta} g^{n+1}|||_{(r,0)}^2 .
    \end{aligned}
$$
Using now \eqref{2.2.1xc} with a choose of $c_0>0$ such that $c_0 C\le \frac {c_1}{6}$,  and induction hypothesis \eqref{2.7a} with small $\epsilon_0$ such that 
$$
C_2 
    \| M_{\delta} g^{n}\|_{H^r_x(L^2_v)}\le C_2C_0\epsilon_0\le \frac{c_1}{6},
$$
then we obtain 
$$
    \begin{aligned}
        \frac{d}{dt}||M_{\delta}g^{n+1}||^2_{H^r_x(L^2_v)}&+  {c_1} ||| M_{\delta} g^{n+1}|||_{(r,0)}^2 \le   \tilde{C}_1||M_{\delta}g^{n+1}||^2_{H^r_x(L^2_v)} .
    \end{aligned}
$$
The Gronwell inequality imply  
$$
\begin{aligned}
\sup_{0\le t\le T}||M_{\delta}g^{n+1}(t)||^2_{H^r_x(L^2_v)}+  {c_1}\int^T_0 ||| M_{\delta} g^{n+1}(s)|||_{(r,0)}^2ds \le {C}_T||M_{\delta}g_0||^2_{H^r_x(L^2_v)} .
\end{aligned}
$$
Using the convergence of approximation solution in \cite{LMPX,three}, we have proven the following results.

\begin{theorem}
Assume that $g\in L^\infty([0, T]; H^r_x(L^2_v))$ is a solution of Cauchy problem \eqref{eq-1} with the initial datum satisfying the asuumption of Theorem \ref{theo1}, let $0<\epsilon_0\ll 1$ and $0<c_0$ small enough, then there exists $C_T>0$ which is independant of $0<\delta<1$, such that 
$$
||M_\delta g(t)||^2_{H^r_x(L^2_v)}+\int^t_0|||M_\delta g(\tau)|||^2_{(r,0)}d\tau\leq C_T\|g_0\|^2_{H^r_x(L^2_v)}.
$$
\end{theorem}
Taking $\delta\ \to\ 0$, we get 
\begin{align*}
||e^{\Psi_s}g(t)||^2_{H^r_x(L^2_v)}+\int^T_0|||e^{\Psi_s}g(\tau)|||^2_{(r,0)}d\tau\leq C_T \|g_0\|^2_{H^r_x(L^2_v)}.
\end{align*}
We have then proven \eqref{2.1} by using the inequality \eqref{Ukai}.

\bigskip\noindent
In the second step, we study now the exponential decay of the solution, 
with the same reason as in the first step, we choose the multiplier
$$
G_\delta(t, v)=\frac{e^{c_0 t\langle v \rangle^{2\tilde{s}}}}{1+\delta e^{c_0 t\langle v \rangle^{2\tilde{s}}}}.
$$
The computation is almost same as for Proposition \ref{proposition2.1}-\ref{proposition2.3}, and we can get 
\begin{align*}
||G_\delta g(t)||^2_{H^r_x(L^2_v)}+\int^T_0|||G_\delta g(\tau)|||^2_{(r,0)}d\tau\leq 2\|g_0\|^2_{H^r_x(L^2_v)}.
\end{align*}

So that the rest of this work is the proof of Proposition \ref{proposition2.2}, Proposition \ref{proposition2.3} and their variation with $ G_\delta(t, v)$. 

\vskip0.5cm

\section{Fourier analysis of Kac Operators}\label{section2}\smallskip

We first recall some analysis results about Kac operators, we have the Bobylev formula for the Fourier transform
of the Kac operator (\cite{even}),
\begin{equation}
{\mathcal{F}}\Big(K(f,\,g)\Big)(\xi )=\frac{1}{2\pi }\int_{-\frac{\pi }{2}}^{%
\frac{\pi }{2}}\beta (\theta )\left\{ \hat{f}(\xi \,\sin {\theta })\hat{g}%
(\xi \,\cos {\theta })-\hat{f}(0)\hat{g}(\xi )\right\} d\theta \,,
\label{2.2}
\end{equation}
By applying the Bobylev formula, \cite{LMPX,three} has proven \eqref{2.2.1xb},  \eqref{2.2.1xc} and the following Proposition. 

\begin{proposition}\label{huak}
Assume that the cross-section satisfies \eqref{1.1+1} with $0<s<1$ and $r>1/2$, it holds 
    \begin{equation}\label{3.4bb}
        |(\mathcal{K}(f,g),h)_{H^r_x(L^2_v)}|\leq||f||_{H^r_x(L^2_v)}
        |||g|||_{(r,0)}
        |||h|||_{(r,0)}
    \end{equation}
    for any suitable function $f,g,h$.
\end{proposition}
Moreover, if we take a trilinear operator $\mathcal{T}$ defined by
$$
\mathcal{T}(f,g, \omega)=\iint \beta(\theta)\omega(v_*)\{f(v_*')g(v')-f(v_*)g(v)\}d\theta dv_*\, ,
$$
that is 
$$
\mathcal{T}(f,g, \mu^{1/2})=\mathcal{K}(f,g)\, ,
$$
We have the next corollary.
\begin{corollary}\label{huagamma}
    If $\omega$ satisfies, for a constant $C$ such that,
    $$
   |\omega(v)|\le C\mu^{1/4}(v),\ \ \ \  \forall\, v\in \mathbb{R}.
    $$
    It still stands
        \begin{equation*}
        |(\mathcal{T}(f,g, \omega), \, h)_{H^r_x(L^2_v)}|\leq C\|f\|_{H^r_x(L^2_v)}
        |||g|||_{(r,0)}
        |||h|||_{(r,0)}
    \end{equation*}
\end{corollary}

The proof of the Corollary can be seen in \cite{hardpotential,trinorm1}.

For the Kac operator $K(f,g)$, there are Bobylev formula about their Fourier transformation, but in our case, it is now the operators $\mathcal{K}(f,g)$, so we study their Fourier transformation.

\begin{lemma}\label{KK+11} For $\mathcal{K}(f,g)$ defined by  \eqref{KKK} and $\beta(\theta)$ satisfies \eqref{1.1+1}, the refined Bobylev formula of $\mathcal{K}$ can be given by
\begin{align*}
\mathcal{F}_{v}\big({\mathcal{K}}(f,g)\big)(\xi)
&=\frac{1}{2\pi}\int_{-\pi /2}^{\pi /2}\int_{\mathbb{R}} \beta (\theta
) \mu^2(u)\{\widehat{f}(u')\, \widehat{g}\left(\xi '\right)-\widehat{f}(u)\widehat{g}(\xi)\}
\,dud\theta
\end{align*}
where
\[
u'= \xi \sin\theta + u \cos\theta,\quad
\xi'= \xi\cos\theta - u \sin\theta.
\]
\end{lemma}
\begin{proof}  We recall 
\begin{align*} 
\mathcal{K}(f,g)(v)&=\mu^{-\frac12}\iint\beta(\theta)(\sqrt{\mu'_*}f_*'\sqrt{\mu'}g'-\sqrt{\mu_*}f_*\sqrt\mu g)d\theta dv_*\\
&=\int_{\mathbb{R}_{v_\ast}}\int_{-\pi /2}^{\pi /2}\beta (\theta
)\left\{ f(v_{\ast }^{\prime })g(v^{\prime })-f(v_{\ast })g(v)\right\}\mu^{1/2}(v_\ast)
d\theta dv_{\ast },
\end{align*}
here we use $\mu(v_*)=\mu^{-1}(v)\mu(v'_*)\mu(v')$.  
Then use Bobylev formula \eqref{2.2}, we get
\begin{align*}
   \mathcal{F}_{v}\left( \mathcal{K}(f,g)\right)(\xi)
&=\frac{1}{2\pi}\int_{-\pi /2}^{\pi /2}\beta (\theta
)\left\{ \Lambda(\xi\cos\theta,\xi\sin\theta)-\widehat{\mu^{1/2}f}(0)\widehat{g}(\xi)\right\}
d\theta ,
\end{align*}
where 
\begin{align*}  \Lambda(\xi\cos\theta,\xi\sin\theta)=\mathcal{F}_{v,v_*}\big(f(v_*)g(v)\big)*\mathcal{F}_{v,v_*}\big(e^{-\frac{|v\sin\theta+v_*\cos\theta|^2}{4}}\big)(\xi\cos\theta,\xi\sin(-\theta)).
\end{align*}
because
\begin{align*}
   &\int_v\int_{v_*}\mu^{1/2}(v_*) f(v')g(v_*')e^{iv\xi}dv_*dv=  \int_v\int_{v_*}\mu^{1/2}(v_*')f(v)g(v_*)e^{iv'\xi}dv_*dv\\
   &= \int_v\int_{v_*}e^{-\frac{|v\sin\theta+v_*\cos\theta|^2}{4}}f(v)g(v_*)e^{-i\xi v\cos\theta-i\xi v_*\sin(-\theta)}dv_*dv\\
   &=\mathcal{F}_{v,v_*}\big(f(v_*)g(v)\big)(\xi_1,\xi_2)*_{\xi_1,\xi_2}\mathcal{F}_{v,v_*}\big(e^{-\frac{|v\sin\theta+v_*\cos\theta|^2}{4}}\big)(\xi_1,\xi_2)\\
   &\qquad\qquad\qquad  *_{\xi_1}\delta(\xi\cos\theta)*_{\xi_2}\delta(\xi\sin(-\theta))\\
   &=\mathcal{F}_{v,v_*}\big(f(v_*)g(v)\big)*\mathcal{F}_{v,v_*}\big(e^{-\frac{|v\sin\theta+v_*\cos\theta|^2}{4}}\big)(\xi\cos\theta,\xi\sin(-\theta))
\end{align*}
We first deal with the $e^{-\frac{|v\sin\theta+v_*\cos\theta|^2}{4}}$ part, consider the following integral:
$$
I(\xi_1, \xi_2) := \iint_{\mathbb{R} \times \mathbb{R}} e^{-\frac{|v\sin\theta+v_*\cos\theta|^2}{4}} e^{-i(v \xi_1 + v_* \xi_2)} \, dv\, dv_*,
$$
where $(v, v_*) \in \mathbb{R}^2$ are integration variables, $(\xi_1, \xi_2)\in \mathbb{R}^2$ are the Fourier variables, and $\theta \in \mathbb{R}$ is a fixed parameter.

To simplify the integral, we perform the following linear change of variables:
\[
\begin{cases}
y = v \sin\theta + v_* \cos\theta, \\
z = -v \cos\theta + v_* \sin\theta,
\end{cases}
\qquad \text{so that} \qquad 
\begin{cases}
v = y \sin\theta - z \cos\theta, \\
v_* = y \cos\theta + z \sin\theta.
\end{cases}
\]
This transformation is orthogonal, so the Jacobian determinant is 1, and the integration measure is preserved: $dv\, dv_* = dy\, dz$.

Substituting into the integral, we obtain:
\[
I(\xi_1, \xi_2) = \iint_{\mathbb{R} \times \mathbb{R}} e^{-\frac{y^2}{4}} e^{-i\left( v \xi_1 + v_* \xi_2 \right)}\, dy\, dz.
\]
Using the inverse transformation expressions for $v$ and $v_*$, we compute:
\begin{align*}
    v \xi_1 + v_* \xi_2 
&= (y \sin\theta - z \cos\theta) \xi_1 + (y \cos\theta + z \sin\theta) \xi_2\\
&= y(\xi_1 \sin\theta + \xi_2 \cos\theta) + z(-\xi_1 \cos\theta + \xi_2 \sin\theta).
\end{align*}
Define:
\[
\alpha := \xi_1 \sin\theta + \xi_2 \cos\theta,\quad
\kappa := -\xi_1 \cos\theta + \xi_2 \sin\theta,
\]
then the integral becomes:
\[
I(\xi_1, \xi_2) = \iint_{\mathbb{R} \times \mathbb{R}} e^{-\frac{y^2}{4}}  e^{-i(y \alpha + z \kappa)}\, dy\, dz.
\]
We now compute the separated integrals,
\[
\int_{\mathbb{R}} e^{-i z \kappa}\, dz = 2\pi\, \delta(\kappa),
\qquad
\int_{\mathbb{R}} e^{-\frac{y^2}{4}} \, e^{-i y \alpha}\, dy = \sqrt{4\pi}\, e^{-\alpha^2},
\]
where $\delta(x)$ is Dirac measure of variable $x$. 
Therefore, the final result is:
\[
I(\xi_1, \xi_2) = 2\pi \sqrt{4\pi} \cdot e^{ - (\xi_1 \sin\theta + \xi_2 \cos\theta)^2 } \cdot \delta\left( -\xi_1 \cos\theta + \xi_2 \sin\theta \right).
\]
or, more compactly,
\[
I(\xi_1, \xi_2) = C\, e^{ - \alpha^2 }\, \delta(\kappa).
\]
We now compute the convolution of the distribution $e^{ - \alpha^2 }\, \delta(\kappa)$ with the tensor product $\widehat{f}(\xi_2)\widehat{g}(\xi_1)$,
\begin{align*}
&\left(e^{ - \alpha^2 }\, \delta(\kappa) * \widehat{f}(\xi_2)\widehat{g}(\xi_1)\right)(\xi_1, \xi_2) \\
&= \int_{\mathbb{R}^2} e^{ - \alpha^2(\zeta_1, \zeta_2) } \delta\big(\kappa(\zeta_1, \zeta_2)\big) 
\widehat{g}(\xi_1 - \zeta_1) \widehat{f}(\xi_2 - \zeta_2)\, d\zeta_1 d\zeta_2.
\end{align*}
Note that the delta function $\delta(\beta(\zeta_1, \zeta_2))$ enforces the condition:
\[
\kappa(\zeta_1, \zeta_2) = -\zeta_1 \cos\theta + \zeta_2 \sin\theta = 0 \quad \Longrightarrow \quad \zeta_2 = \zeta_1 \frac{\cos\theta}{\sin\theta}.
\]
Substituting this into $\alpha$, we get:
\[
\alpha = \zeta_1 \sin\theta + \zeta_2 \cos\theta = \zeta_1 \sin\theta + \zeta_1 \frac{\cos^2\theta}{\sin\theta} = \zeta_1 \frac{1}{\sin\theta}.
\]
Thus, the convolution becomes:
\[
\begin{aligned}
& \iint_{\mathbb{R}^2} e^{ - \alpha^2(\zeta_1, \zeta_2) } \delta(\beta(\zeta_1, \zeta_2)) \widehat{g}(\xi_1 - \zeta_1)\widehat{f}(\xi_2 - \zeta_2) \, d\zeta_1 d\zeta_2 \\
&= \int_{\mathbb{R}} e^{ - \left( \frac{\zeta_1}{\sin\theta} \right)^2 } \widehat{g}(\xi_1 - \zeta_1)\, \widehat{f}\left(\xi_2 -\zeta_1 \frac{\cos\theta}{\sin\theta}\right)\frac{1}{\sin\theta} \, d\zeta_1.\\
&= \int_{\mathbb{R}} e^{ - u^2 } \widehat{g}(\xi_1 - u\sin\theta)\, \widehat{f}\left(\xi_2 - u\cos\theta\right)\,du.
\end{aligned}
\]
Then let $\xi_1=\xi\cos\theta\,,\xi_2=-\xi \sin\theta$, and variable change $(u,\theta)\to(-u,-\theta)$ for the gain part,
\begin{align*}
   \mathcal{F}_{v}\big({\mathcal{K}}(f,g)\big)(\xi)
&=\frac{1}{2\pi}\int_{-\pi /2}^{\pi /2}\beta (\theta
)\{ (\int_{\mathbb{R}} \mu^2(u) \widehat{g}(\xi_1 - u\sin\theta)\, \widehat{f}\left(\xi_2 
+ u\cos\theta\right)\,du)
\\&\qquad\qquad\qquad\qquad-\widehat{\mu^{1/2}f}(0)\widehat{g}(\xi)\}
d\theta \\
&=\frac{1}{2\pi}\int_{-\pi /2}^{\pi /2}\beta (\theta
)\{ \int_{\mathbb{R}} \mu^2(u) \widehat{g}(\xi\cos\theta - u\sin\theta)\, \widehat{f}\left(-\xi \sin\theta
- u\cos\theta\right)
\\&\qquad\qquad\qquad\qquad-\mu^2(u)\widehat{f}(u)\widehat{g}(\xi)\}
\,dud\theta\\
&=\frac{1}{2\pi}\int_{-\pi /2}^{\pi /2}\beta (\theta
)\{ \int_{\mathbb{R}} \mu^2(u) \widehat{g}(\xi\cos\theta - u\sin\theta)\, \widehat{f}\left(\xi \sin\theta
+ u\cos\theta\right)
\\&\qquad\qquad\qquad\qquad-\mu^2(u)\widehat{f}(u)\widehat{g}(\xi)\}
\,dud\theta.
\end{align*}
\end{proof}

\begin{corollary}\label{Kx}
If we take $x$ into consideration, the refined Bobylev formula can be conclude in
 \begin{align*}
   \mathcal{F}_{x,v}\big({\mathcal{K}}(f,g)\big)(\xi,\eta)
&=\frac{1}{2\pi}\int_{-\pi /2}^{\pi /2} \int_{\mathbb{R}} \beta (\theta
)\mu^2(u)\Big\{  \widehat{f}(t,\eta,u')*_{\eta} \widehat{g}\left(t,\eta,\xi '\right)\\
&\qquad\qquad\qquad- \widehat{f}(t,\eta,u)*_{\eta}\widehat{g}(t,\eta,\xi)\Big\}
\,dud\theta.\\
&=\frac{1}{2\pi}\int_{-\pi /2}^{\pi /2}\int_{\mathbb{R}^2}\beta (\theta
)  \mu^2(u) \Big\{\widehat{f}(t,\eta-\eta_1,u') \widehat{g}\left(t,\eta_1,\xi '\right)\\
&\qquad\qquad\qquad- \widehat{f}(t,\eta-\eta_1,u)\widehat{g}(t,\eta_1,\xi)\Big\}
\,dud\eta_1 d\theta.
\end{align*}
Furthermore, for $\alpha>\frac{1}{4}$,
\begin{align*}
\mathcal{F}_{x,v}\big({\mathcal{T}}(f,g,\mu^{\alpha})\big)(\xi,\eta)
&=\frac{1}{2\pi}\int_{-\pi /2}^{\pi /2}\int_{\mathbb{R}^2}\beta (\theta
) \mu^{\frac{1}{\alpha}}(u)\Big\{\widehat{f}(t,\eta-\eta_1,u') \widehat{g}\left(t,\eta_1,\xi '\right)\\
&\qquad\qquad\qquad\qquad -\widehat{f}(t,\eta-\eta_1,u)\widehat{g}(t,\eta_1,\xi)\Big\}
\,dud\eta_1 d\theta.
\end{align*}

\end{corollary}

Now prepare some computation for the weighted function.
\begin{lemma}\label{M} Let $M_{\delta}$ the weighted function defined in \eqref{DM}, then
\[
|\partial_\xi M_\delta(t,\eta,\xi)|
\leq \ C \,M_{\delta}(t,\eta,\xi), 
\]
and
\[
|\partial_\xi^2 M_\delta(t,\eta,\xi)|
\leq  C \,M_{\delta}(t,\eta,\xi). 
\]
\end{lemma}
Here we just need 
$$
|\partial_\xi \Psi_s(t,\eta,\xi)|\le C.
$$
The next estimation is a preparation for the nonlinear part. 
\begin{lemma}\label{7num}
    For any $\tau\in[0,1]$, $M_\delta$ defined by \eqref{DM}, then we have 
    \begin{align*}
         M_\delta(t,\eta,\xi_{\tau})\le C M_\delta(t,\eta_1,\xi')\max\{ M_\delta(t,\eta-\eta_1,u'), M_\delta(t,\eta-\eta_1,-u')\}e^{\langle u\rangle^{2\tilde{s}}},
    \end{align*}
    where $\xi_\tau=\xi'-\tau(\xi'-\xi)$.
\end{lemma}

\begin{proof}
We have firstly
\begin{align*}
    \Psi_s(t,\eta,-{t\eta}-\xi)= \Psi_s(t,\eta,\xi)
\end{align*}
That means $\Psi_s(t,\eta,\xi)$ is symmetric with respect to $\xi=\frac{t\eta}{2}$.
Now notice
\begin{align*}
  (\xi-\xi')&=\xi(1-\cos\theta)+u\sin\theta\\
  &=\sin\frac{\theta}{2}\, \frac{\xi\, \sin\theta+u(1+\cos\theta)}{\cos\frac{\theta}{2}}=\sin\frac{\theta}{2}\, \frac{u'+u}{\cos\frac{\theta}{2}},
\end{align*}
then 
\begin{align*}
    \langle \xi_{\tau}+\rho\eta\rangle^{2\tilde{s}}&\le \langle \xi'+\rho\eta+\tau\tan{\frac{\theta}{2}}\ (u+u')\rangle^{2\tilde{s}}\\
&\le\langle \xi'+\rho\eta_1+\rho(\eta-\eta_1)+\tau\tan{\frac{\theta}{2}}\ (u+u')\rangle^{2\tilde{s}}\\
&\le \langle \xi'+\rho\eta_1\rangle^{2\tilde{s}}+\langle\rho(\eta-\eta_1)+\tau\tan{\frac{\theta}{2}}\ (u')\rangle^{2\tilde{s}}+\langle \tau\tan{\frac{\theta}{2}}\ (u)\rangle^{2\tilde{s}}.
\end{align*}
On the other hand 
$$
\langle\rho(\eta-\eta_1)+\tau\tan{\frac{\theta}{2}}\ (u')\rangle^{2\tilde{s}}\le \max\{\langle\rho(\eta-\eta_1)-u'\rangle^{2\tilde{s}},\langle\rho(\eta-\eta_1)+u'\rangle^{2\tilde{s}}\}.
$$
Since $\frac{x}{1+\delta x}$ is a increasing function for $x>0$, we have
\begin{align*}
    \frac{ABC}{1+\delta ABC}&\le 9\frac{A}{1+\delta A}\frac{B}{1+\delta B}\frac{C}{1+\delta C}\\
    &\le 9\frac{A}{1+\delta A}\frac{B}{1+\delta B}{C},
\end{align*}
then we have 
\begin{align*}
M_\delta(t,\eta,\xi_{\tau})\le C M_\delta(t,\eta_1,\xi')\max
\{ M_\delta(t,\eta-\eta_1,u'), M_\delta(t,\eta-\eta_1,-u')\}e^{\langle u\rangle^{2\tilde{s}}}.
    \end{align*}
\end{proof}

\begin{lemma}\label{bd}
Let 
\[
\begin{aligned}
v'&=v\cos\theta-v_{*}\sin\theta,\\
v_{*}'&=v\sin\theta+v_{*}\cos\theta,
\end{aligned}
\qquad \theta\in[0,2\pi).
\]
Then, for every triple $(v,v_{*},\theta)$ one has, for $0<s\le 1$,
\begin{equation*}
2^{s-1}\;\le\;
\frac{\langle v_{*}\rangle^{2s}+\langle v\rangle^{2s}}
{\langle v_{*}'\rangle^{2s}+\langle v'\rangle^{2s}}
\;\le\;2^{\,1-s}\,.  
\end{equation*}
In particular, when $s=1$ the ratio equals $1$ exactly. 
\end{lemma}


\section{Analysis of linear operator}\label{section3}\smallskip

We now come back to the proof  the Proposition \ref{proposition2.2},  we first have
$$
\big(\mathcal{L}g, M_{\delta}^2g \big)_{H^r_x(L^2_v)}=
\big(\mathcal{L}( M_{\delta} g), M_{\delta}g \big)_{H^r_x(L^2_v)}
 + \big([\mathcal{L},  M_{\delta}] g, M_{\delta}g \big)_{H^r_x(L^2_v)}.
 $$
Firstly
$$
\big(\mathcal{L}( M_{\delta} g), M_{\delta}g \big)_{H^r_x(L^2_v)}
   \ge \tilde{c}_1 ||| M_{\delta} g|||_{(r,0)}^2 -{C}_1||M_{\delta}g||^2_{H^r_x(L^2_v)}.     
$$
So that the proof of  the Proposition \ref{proposition2.2} is transfer to the following estimates of commutators.

\begin{proposition}\label{pr1}
Assume that the cross-section satisfies \eqref{1.1+1} with $0<s<1$, there exists $C>0$ such that for suitable function $g, h$,
$$
\Big|\Big([\mathcal{L},\  M_{\delta}]g,h \Big)_{H^r_x(L^2_v)}\Big|
\leq C \| M_{\delta} g\|_{H^r_x(L^2_v)}\|h\|_{H^r_x(L^2_v)}
$$
\end{proposition}
We prove this proposition by the following two Lemmas.
\begin{lemma}\label{2.4}
Assume that the cross-section satisfies \eqref{1.1+1} with $0<s<1$ and $r>1/2$, there exists $C>0$ such that for $0<t<T, v\in \mathbb{R}, x\in\mathbb{R}$,
\begin{equation*}
  \Big|\Big([\mathcal{L}_1,\  M_{\delta}]g, h \Big)_{H^r_x(L^2_v)}\Big|
\leq C\|M_{\delta}g\|_{H^r_x(L^2_v)}\|h\|_{H^r_x(L^2_v)}.
\end{equation*}
\end{lemma}
\begin{proof}
We have
\begin{align*}
&\Big([\mathcal{L}_1,\  M_{\delta}]g,h \Big)_{H^r_x(L^2_v)}= \Big(M_{\delta} \mathcal{K}(\mu^{1/2},  g),  h \Big)_{H^r_x(L^2_v)}-
\Big(\mathcal{K}(\mu^{1/2},  M_{\delta}g), h \Big)_{H^r_x(L^2_v)}.
\end{align*}
Using Plancerel formula, Lemma \ref{KK+11}, and a cut-off approximation procesus, we can formaly wirte the collisian term in two parts by following steps:
\begin{align*}
&\Big([\mathcal{L}_1,\  M_{\delta}]g,h \Big)_{H^r_x(L^2_v)}
\\ &= \frac{1}{2\pi}\Big\{\int_{\mathbb{R}^2_{\xi, \eta}} \langle\eta\rangle^{2r}\int_{-\pi /2}^{\pi /2}\beta (\theta
)\Big\{ \int_{\mathbb{R}} \mu^2(u)M_{\delta}(t,\eta,\xi)\widehat{g}(t,\eta,\xi')\, \widehat{\mu^{\frac{1}{2}}}\left(u '\right)
\,du\Big\}d\theta\,\overline{\widehat{h}    }\,du d\xi d\eta\\
&\ \ \ \ -\int_{\mathbb{R}^2_{\xi, \eta}}\langle\eta\rangle^{2r}\int_{-\pi /2}^{\pi /2}\beta (\theta
)\Big\{ \int_{\mathbb{R}} \mu^2(u)M_{\delta}(t,\eta,\xi)\widehat{g}(t,\eta,\xi)\, \widehat{\mu^{\frac{1}{2}}}\left(u\right)
\,du\Big\}d\theta\,\overline{\widehat{h}    }\,du d\xi d\eta\\
&-\int_{\mathbb{R}^2_{\xi, \eta}} \langle\eta\rangle^{2r}\int_{-\pi /2}^{\pi /2}\beta (\theta
)\Big\{ \int_{\mathbb{R}} \mu^2(u)M_{\delta}(t,\eta,\xi')\widehat{g}(t,\eta,\xi')\, \widehat{\mu^{\frac{1}{2}}}\left(u '\right)
\,du\Big\}d\theta\,\overline{\widehat{h}    }\,du d\xi d\eta\\
&\ \ \ \ +\int_{\mathbb{R}^2_{\xi, \eta}}\langle\eta\rangle^{2r}\int_{-\pi /2}^{\pi /2}\beta (\theta
)\Big\{ \int_{\mathbb{R}} \mu^2(u)M_{\delta}(t,\eta,\xi)\widehat{g}(t,\eta,\xi)\, \widehat{\mu^{\frac{1}{2}}}\left(u\right)
\,du\Big\}d\theta\,\overline{\widehat{h}    }\,du d\xi d\eta\Big\}.
\end{align*}
Thus, we have
\begin{align*}
&\Big([\mathcal{L}_1,\  M_{\delta}]g,h \Big)_{H^r_x(L^2_v)}
\\ &\qquad= \frac{1}{2\pi}\int_{\mathbb{R}^2_{\xi, \eta}} \langle\eta\rangle^{2r}\int_{-\pi /2}^{\pi /2}\beta (\theta
)\Big\{\int_{\mathbb{R}} \mu^2(u)\big(M_{\delta}(t,\eta,\xi)-M_{\delta}(t,\eta,\xi')\big)\\
&\qquad\qquad\qquad\qquad \times\widehat{g}(t,\eta,\xi')\, \widehat{\mu^{\frac{1}{2}}}\left(u '\right)
\,du \Big\}d\theta\ \overline{\widehat{h}    }\, d\xi d\eta
\end{align*}
{\bf The case of  $0<s<1/2$ :}\ 
The first-order Taylor formula gives
\begin{align*}
M_{\delta}(t,\eta,\xi)-M_{\delta}(t,\eta,\xi')=
\int_0^1(\xi-\xi')\partial_{\xi}M_{\delta}(t,\eta,\xi_\tau)d\tau
\end{align*}
where $\xi_\tau=\xi'-\tau(\xi'-\xi)$. Thus we have
\begin{align*}
&\Big|\Big([\mathcal{L}_1,\  M_{\delta}]g,h \Big)_{H^r_x(L^2_v)}\Big|
\\
&\qquad=C \Big|\int_{\mathbb{R}^2_{\xi, \eta}} \langle\eta\rangle^{2r}\int_{-\pi /2}^{\pi /2}\beta (\theta
)\Big\{ \int_{\mathbb{R}} \mu^2(u)\big(M_{\delta}(t,\eta,\xi)-M_{\delta}(t,\eta,\xi')\big)\\
&\qquad\qquad\qquad\qquad\qquad\times \widehat{g}(t,\eta,\xi')\, \mu^2\left(u '\right)
\,du\Big\}d\theta \ \overline{\widehat{h}    }\, d\xi d\eta\Big|\\
&\qquad=C \Big|\int_{\mathbb{R}^3_{\xi, \eta, u}}\int_{-\pi /2}^{\pi /2} \langle\eta\rangle^{2r}\beta (\theta
)\Big\{  \mu^2(u)\left(\int_0^1(\xi-\xi')\partial_\xi M_{\delta}(t,\eta,\xi_\tau)d\tau\right)\\
&\qquad\qquad\qquad\qquad\qquad\times \widehat{g}(t,\eta,\xi')\,\mu^2\left(u '\right)
\,\Big\} \ \overline{\widehat{h}    }\,du d\xi d\theta  d\eta\Big|.
\end{align*}
 Note that for any $\xi,\xi'$, there is
\begin{align*}
    \langle \xi_{\tau}+\rho\eta\rangle&\le \langle \xi'+\rho\eta\rangle+|\xi_{\tau}-\xi'|=\langle \xi'+\rho\eta\rangle+|\tau(\xi-\xi')|\\
&\le\langle \xi'+\rho\eta\rangle+|\tan{\frac{\theta}{2}}|\ |u+u'|,\,\,\forall u,\xi\in\RR.
\end{align*}
Thus we have 
\begin{align*}
     \frac{M_{\delta}(t, \eta, \xi_{\tau})}{M_{\delta}(t, \eta, \xi')}\le 3e^{t\tan{\frac{\theta}{2}}|u+u'|},
\end{align*}
Using $|u+u'|\le\langle u\rangle+\langle u'\rangle\le 4\langle u\rangle\langle u'\rangle$ and Lemma \ref{M},
\begin{align*}
&\Big|\int_{\mathbb{R}^3_{\xi, \eta, u}}\int_{-\pi /2}^{\pi /2}\langle\eta\rangle^{2r}\beta (\theta
) \mu^2(u)\left(\int_0^1(\xi-\xi')\partial_\xi M_{\delta}(t,\eta,\xi_\tau)d\tau\right)\\
&\qquad\qquad\qquad\times \widehat{g}(t,\eta,\xi')\,\mu^2\left(u '\right)
\ \overline{\widehat{h}}\,du\, d\xi d\theta  d\eta\Big|\\
&= C\Big|\int_{\mathbb{R}^3_{\xi, \eta, u}}\int_{-\pi /2}^{\pi /2} \langle\eta\rangle^{2r}\beta (\theta
)  \mu^2(u)\big(\int_0^1\sin\frac{\theta}{2}\frac{u+u'}{\cos(\theta/2)}\partial_\xi M_{\delta}(t,\eta,\xi_\tau)d\tau)\big)\\
&\qquad\qquad\qquad\times \widehat{g}(t,\eta,\xi')\, \mu^2\left(u '\right)
\ \overline{\widehat{h}}\,du d\xi d\theta  d\eta\Big|\\
&\le C\int_{\mathbb{R}^3_{\xi, \eta, u}}\int_{-\pi /2}^{\pi /2} \langle\eta\rangle^{2r}\beta (\theta
)  \mu^2(u)\Big|\sin\frac{\theta}{2}\frac{\langle u\rangle\langle u'\rangle}{\cos(\theta/2)}e^{t\tan{\frac{\theta}{2}}|u+u'|}\Big|M_{\delta}(t, \eta, \xi')\\
&\qquad\qquad\qquad\times \Big|\widehat{g}(t,\eta,\xi')\, \mu^2\left(u '\right)
 \,\overline{\widehat{h}    }\Big|\,du d\xi d\theta  d\eta.
\end{align*}
By H\"older inequality and $\beta(\theta)\sin{\theta}\sim \theta^{-2s}$ which converge for $0<s<1/2$.
\begin{align*}
&\int_{\mathbb{R}^3_{\xi, \eta, u}}\int_{-\pi /2}^{\pi /2} \langle\eta\rangle^{2r}\beta (\theta
)\mu^2(u)\big|\sin\frac{\theta}{2}\frac{\langle u\rangle\langle u'\rangle}{\cos(\theta/2)}e^{t\tan{\frac{\theta}{2}}(|u|+|u'|)}M_{\delta}(t, \eta, \xi')\big|\\
&\qquad\qquad\qquad\times \big|\widehat{g}(t,\eta,\xi')\, \mu^2\left(u '\right)
\, \overline{\widehat{h}    }\big|\,du d\xi d\theta  d\eta\\
&\leq \Big\{\int_{\mathbb{R}^3_{\xi, \eta, u}}\int_{-\pi /2}^{\pi /2} \beta (\theta
)  \big|\sin\frac{\theta}{2}\frac{\langle u'\rangle^2}{\cos(\theta/2)}\big|\ \big|e^{t\tan{\frac{\theta}{2}}(|u'|)}M_{\delta}(t, \eta, \xi')\big|^2\\
&\qquad\qquad\qquad\qquad\times\big| \langle\eta\rangle^{r} \widehat{g}(t,\eta,\xi')\, \mu^2\left(u '\right)
\,|^2\,du d\xi d\theta  d\eta\Big\}^{1/2}\\
&\qquad\qquad\qquad\times \Big\{\int_{\mathbb{R}^3_{\xi, \eta, u}}\int_{-\pi /2}^{\pi /2} \beta (\theta
)  \big|\sin\frac{\theta}{2}\frac{\langle u\rangle^2}{\cos(\theta/2)}\big|\\
&\qquad\qquad\qquad\qquad\times \big|\mu^2(u) e^{t\tan{\frac{\theta}{2}}(|u|)}
\overline{\langle\eta\rangle^{r}\widehat{h}    }|^2\,du d\xi d\theta  d\eta\Big\}^{1/2}.
\end{align*}
For the first term, use the variable change $(\xi',u')\to(\xi,u)$, this transformation is orthogonal, so the Jacobian determinant is $1$
\begin{align*}
   & \Big\{\int_{\mathbb{R}^3_{\xi, \eta, u}}\int_{-\pi /2}^{\pi /2} \langle\eta\rangle^{2r}\beta (\theta
)  \sin\frac{\theta}{2}\frac{\langle u'\rangle^2}{\cos(\theta/2)}\big|e^{t\tan{\frac{\theta}{2}}(|u'|)}M_{\delta}(t, \eta, \xi')\big|^2\\
&\qquad\qquad\qquad\qquad\times |\widehat{g}(t,\eta,\xi')\, \mu^2\left(u '\right)
\,|^2\,du d\xi d\theta  d\eta\Big\}^{1/2}\\
&\leq C\| \langle\eta\rangle^{r} M_{\delta}(t,\eta,\xi)\widehat{g}(t,\eta,\xi)\|_{L^2_\eta(L^2_\xi)}\\
&\qquad\qquad\qquad\times\|\mu^2(u)\langle u\rangle e^{|u|}\|_{L^2_u}\|\theta^{-2s}\|^{\frac{1}{2}}_{L^1_\theta}.
\end{align*}
For the second term,
\begin{align*}
&\Big\{\int_{\mathbb{R}^3_{\xi, \eta, u}}\int_{-\pi /2}^{\pi /2} \beta (\theta
) \big(\sin\frac{\theta}{2}\frac{\langle u\rangle^2}{\cos(\theta/2)}\\
&\qquad\qquad\qquad\qquad\times | \mu^2(u) e^{t\tan{\frac{\theta}{2}}(|u|)}
\langle\eta\rangle^{r}\overline{\widehat{h}    }|^2\,du d\xi d\theta  d\eta\Big\}^{1/2}\\
&\leq \| \langle\eta\rangle^{r}{\widehat{h}}(t,\eta,\xi)\|_{L^2_\eta(L^2_\xi)}\|\mu^2(u)\langle u\rangle e^{|u|}\|_{L^2_u}\|\theta^{-2s}\|^{\frac{1}{2}}_{L^1_\theta}.
\end{align*}
Hence, for $0<s<\frac 12$, we get
\begin{align*}
 &\Big|\Big([\mathcal{L}_1,\  M_{\delta}(t, D_x, D_v)]g, h \Big)_{H^r_x(L^2_v)}\Big|\\
&\leq C\| M_{\delta}g \|_{H^r_x(L^2_v)}\| h\|_{H^r_x(L^2_v)}.
\end{align*}
{\bf The case of $1/2\le s<1$:}, We use now the second-order Taylor formula.
\begin{align*}
M_{\delta}(t,\eta,\xi)-M_{\delta}(t,\eta,\xi')&=(\xi-\xi')\partial_\xi M_{\delta}(t,\eta,\xi')\\
&\qquad+\int_0^1(1-\tau)(\xi-\xi')^2\partial_x^2M_{\delta}(t,\eta,\xi_\tau)d\tau,
\end{align*}
and for a suitable function $F$, for any fixed $u\in \RR$
\[
\int_{\mathbb{R}}\int^{\pi/2}_{-\pi/2}\beta(\theta)(\xi-\xi')F(\xi')d\theta d\xi=0.
\]
Thus we have
\begin{align*}
&\Big|\Big([\mathcal{L}_1,\  M_{\delta}]g,h \Big)_{H^r_x(L^2_v)}\Big|
\\
&\qquad=C \Big|\int_{\eta}\int_{\xi} \langle\eta\rangle^{r}\int_{-\pi /2}^{\pi /2}\beta (\theta
)\Big\{ \int_{\mathbb{R}} \mu^2(u)\big(M_{\delta}(t,\eta,\xi)-M_{\delta}(t,\eta,\xi')\big)\\
&\qquad\qquad\qquad\times \widehat{g}(t,\eta,\xi')\, \mu^2\left(u '\right)
\,du\Big\}d\theta \,\overline{\langle\eta\rangle^{r}\widehat{h}    }\, d\xi d\eta\Big|\\
&\qquad=C \Big|\int_{\mathbb{R}^3_{\xi, \eta, u}}\int_{-\pi /2}^{\pi /2}  \langle\eta\rangle^{r}\beta (\theta
)\Big\{  \mu^2(u)\Big(\int_0^1(1-\tau)(\xi-\xi')^2\partial_{\xi}^2M_{\delta}(t,\eta,\xi_\tau)d\tau \Big)\\
&\qquad\qquad\qquad\times \widehat{g}(t,\eta,\xi')\, \mu^2\left(u '\right)
\,\Big\} \,\overline{\langle\eta\rangle^{r}\widehat{h}    }\,d\theta du d\xi  d\eta\Big|.
\end{align*}
Thanks to Lemma \ref{M} we have
\begin{align*}
    \Big|&\int_{\mathbb{R}^3_{\xi, \eta, u}}\int_{-\pi /2}^{\pi /2}  \langle\eta\rangle^{r}\beta (\theta
)  \mu^2(u)\Big(\int_0^1(1-\tau)(\xi-\xi')^2\partial_{\xi}^2M_{\delta}(t,\eta,\xi_\tau)d\tau \Big)\\
&\qquad\qquad\qquad\times \widehat{g}(t,\eta,\xi')\, \mu^2\left(u '\right) \,\overline{\langle\eta\rangle^{r}\widehat{h}    }\,d\theta du d\xi  d\eta\Big|\\
&\le C \int_{\mathbb{R}^3_{\xi, \eta, u}}\int_{-\pi /2}^{\pi /2}  \langle\eta\rangle^{r}\beta (\theta
)  \mu^2(u)\Big|\sin^2\frac{\theta}{2}\frac{(u+u')^2}{\cos^2(\theta/2)}e^{t\tan{\frac{\theta}{2}}|u+u'|}M_{\delta}(t, \eta, \xi')\Big|\\
&\qquad\qquad\qquad\times |\widehat{g}(t,\eta,\xi)\, \mu^2\left(u '\right)
\, \,\overline{\langle\eta\rangle^{r}\widehat{h}   }|\,d\theta du d\xi  d\eta.
\end{align*}
Now $|\xi-\xi'|^2=C\sin^2\frac{\theta}{2}|u+u'|^2$, 
Using H\"older inequality, and  $|u+u'|^2\le C\langle u\rangle^2\langle u'\rangle^2 $, we have
\begin{align*}
    &\int_{\mathbb{R}^3_{\xi, \eta, u}}\int_{-\pi /2}^{\pi /2}  \langle\eta\rangle^{r}\beta (\theta
) \mu^2(u)\Big|\sin^2\frac{\theta}{2}\frac{(u+u')^2}{\cos^2(\theta/2)}e^{t\tan{\frac{\theta}{2}}|u+u'|}M_{\delta}(t, \eta, \xi')\Big|\\
&\qquad\qquad\qquad\times |\widehat{g}(t,\eta,\xi)\, \mu^2\left(u '\right) \,\overline{\langle\eta\rangle^{r}\widehat{h}    }|\,du d\xi d\theta  d\eta\\
&\leq C\Big\{\int_{\mathbb{R}^3_{\xi, \eta, u}}\int_{-\pi /2}^{\pi /2}  \beta (\theta
)  \sin^2\frac{\theta}{2}\frac{\langle u'\rangle^4}{\cos(\theta/2)}\Big|e^{t\tan{\frac{\theta}{2}}(|u'|)}M_{\delta}(t, \eta, \xi')\\
&\qquad\qquad\qquad\qquad\times \langle\eta\rangle^{r}\widehat{g}(t,\eta,\xi')\, \mu^2\left(u '\right)
\ \Big|^2\,du d\xi d\theta  d\eta\Big\}^{1/2}\\
&\qquad\qquad\qquad\times \Big\{\int_{\mathbb{R}^3_{\xi, \eta, u}}\int_{-\pi /2}^{\pi /2}  \beta (\theta
)  \mu^2(u)\sin^2\frac{\theta}{2}\frac{\langle u\rangle^4}{\cos(\theta/2)}\\
&\qquad\qquad\qquad\qquad\times |e^{t\tan{\frac{\theta}{2}}(|u|)}
\overline{\langle\eta\rangle^{r}\widehat{h}    }|^2\,du d\xi d\theta  d\eta\Big\}^{1/2}.
\end{align*}
For the first term, use the variable change $(\xi',u')\to(\xi,u)$, this transformation is orthogonal, so the Jacobian determinant is $1$,
\begin{align*}
   & \Big\{\int_{\mathbb{R}^3_{\xi, \eta, u}}\int_{-\pi /2}^{\pi /2}   \langle\eta\rangle^{2r}\beta (\theta
)  \sin^2\frac{\theta}{2}\frac{\langle u'\rangle^4}{\cos^2(\theta/2)}\Big|e^{t\tan{\frac{\theta}{2}}(|u'|)}M_{\delta}(t, \eta, \xi')\\
&\qquad\qquad\qquad\qquad\times \widehat{g}(t,\eta,\xi')\, \mu^2\left(u '\right)
\,\Big|^2\,du d\xi d\theta  d\eta\Big\}^{1/2}\\
=&  \Big\{\int_{\mathbb{R}^3_{\xi, \eta, u}}\int_{-\pi /2}^{\pi /2}  \langle\eta\rangle^{2r}\beta (\theta
)  \sin^2\frac{\theta}{2}\frac{\langle u\rangle^4}{\cos^2(\theta/2)}\Big|e^{t\tan{\frac{\theta}{2}}(|u|)}M_{\delta}(t, \eta, \xi)\\
&\qquad\qquad\qquad\qquad\times \widehat{g}(t,\eta,\xi)\, \mu^2\left(u \right)
\,\Big|^2\,du d\xi d\theta  d\eta\Big\}^{1/2}\\
&\leq C\| \langle\eta\rangle^{r} M_{\delta}(t,\eta,\xi)\widehat{g}(t,\eta,\xi)\|_{L^2_\eta(L^2_\xi)}\\
&\qquad\qquad\qquad\times\|\mu^2(u)\langle u\rangle^2 e^{|u|}\|_{L^2_u}\|\theta^{1-2s}\|^{\frac{1}{2}}_{L^1_\theta}.
\end{align*}
For the second term,
\begin{align*}
     &\Big\{\int_{\mathbb{R}^3_{\xi, \eta, u}}\int_{-\pi /2}^{\pi /2}  \beta (\theta
) \mu^2(u) \sin^2\frac{\theta}{2}\frac{\langle u\rangle^4}{\cos^2(\theta/2)}\\
&\qquad\qquad\qquad\qquad\times | e^{t\tan{\frac{\theta}{2}}(|u|)}
\overline{\langle\eta\rangle^{r}\widehat{h}    }|^2\,du d\xi d\theta  d\eta\Big\}^{1/2}\\
&\leq C\| \langle\eta\rangle^{r}{\widehat{h}}(t,\eta,\xi)\|_{L^2_\eta(L^2_\xi)}\|\mu^2(u)\langle u\rangle^2 e^{|u|}\|_{L^2_u}\|\theta^{1-2s}\|^{\frac{1}{2}}_{L^1_\theta}.
\end{align*}
Hence, for $\frac 12\le s<1$, we also get
$$ 
 \Big|\Big([\mathcal{L}_1,\  M_{\delta}]g, h \Big)_{H^r_x(L^2_v)}\Big|
\leq C\| M_{\delta}g\|_{H^r_x(L^2_v)}\| h\|_{H^r_x(L^2_v)}.
$$
which ends the proof of Lemma \ref{2.4}.
\end{proof}

The commutators with $\mathcal{L}_2$ is simalry.
\begin{lemma}\label{2.4.1}
Assume that the cross-section satisfies \eqref{1.1+1} with $0<s<1$ and $r>1/2$, there exists $C>0$ such that
\begin{equation*}
\Big|\Big([\mathcal{L}_2,\  M_{\delta}]g,h \Big)_{H^r_x(L^2_v)}\Big|
\leq C\|M_{\delta} g\|_{H^r_x(L^2_x)}\| h\|_{H^r_x(L^2_x)}.
\end{equation*}
\end{lemma}
We omit the proof of this Lemma.

\section{Upper estimates of non linear operator} 

Now we turn attention to Proposition \ref{proposition2.3}, we have, by using \eqref{3.4bb}, with the aid of refined Bobylev formula of Corollary \ref{Kx},  
\begin{align*}
&\Big( M_{\delta}\mathcal{K}(f,\ g),\  h \Big)_{H^r_x(L^2_v)}=\Big( \langle \eta\rangle^r\mathcal{F}(M_{\delta}\mathcal{K}(f,\ g)),\  \langle \eta\rangle^r\mathcal{F}(h) \Big)_{L^2_{\eta, \xi}}\\
&=\int_{\mathbb{R}^4}\int_{-\frac{\pi}{2}}^{\frac{\pi}{2}}  \beta(\theta)\ M_{\delta}(t, \eta, \xi)\mu^{2}(u)\langle \eta\rangle^r 
\Big(\widehat{f}(t,\eta-\eta_1,u') \widehat{g}\left(t,\eta_1,\xi '\right)\\
&\qquad\qquad\qquad  -\widehat{f}(t,\eta-\eta_1,u) \widehat{g}\left(t,\eta_1,\xi \right)\Big)\,\overline{\langle \eta\rangle^r\widehat{h}(t,\eta,\xi) }\,d\theta d{\eta_1} du  d\xi d\eta\\
&=\int_{\mathbb{R}^4}\int_{-\frac{\pi}{2}}^{\frac{\pi}{2}} \beta(\theta)\ \langle \eta\rangle^{2r} \Big(M_{\delta}(t, \eta, \xi)\mu^{2}(u)
\widehat{f}(t,\eta-\eta_1,u') \widehat{g}\left(t,\eta_1,\xi '\right)\overline{\widehat{h}(t,\eta,\xi) }\\
&\qquad\qquad -M_{\delta}(t, \eta, \xi')\mu^{2}(u')\widehat{f}(t,\eta-\eta_1,u') \widehat{g}\left(t,\eta_1,\xi '\right)\overline{\widehat{h}(t,\eta,\xi') }\Big)\,\,d\theta d{\eta_1} du  d\xi d\eta,
\end{align*}
We get then
\begin{align*}
&\Big( M_{\delta}\mathcal{K}(f,\ g),\  h \Big)_{H^r_x(L^2_v)}\\
&=\int_{\mathbb{R}^4}\int_{-\frac{\pi}{2}}^{\frac{\pi}{2}}
\beta(\theta)\ \langle \eta\rangle^{2r} \widehat{f}(t,{\eta_1}, u')\widehat{g}(t,\eta-{\eta_1}, \xi')\Big(M_{\delta}(t, \eta, \xi)\mu^{2}(u)
\overline{\widehat{M_{\delta} h}(t,\eta,\xi) }\\
&\qquad\qquad\qquad\qquad -M_{\delta}(t, \eta, \xi')\mu^{2}(u')\overline{\widehat{ h}(t,\eta,\xi') }\Big)\,\,d\theta d{\eta_1} du  d\xi d\eta\\
\\
&=\int_{\mathbb{R}^4}\int_{-\frac{\pi}{2}}^{\frac{\pi}{2}} \beta(\theta)\ \langle \eta\rangle^{2r} \widehat{f}(t,{\eta_1}, u')\widehat{g}(t,\eta-{\eta_1}, \xi')\Big(M_{\delta}(t, \eta, \xi)\mu^{2}(u)
\overline{\widehat{ h}(t,\eta,\xi) }\\
&\qquad\qquad -M_{\delta}(t, \eta, \xi')\mu^{2}(u)
\overline{\widehat{ h}(t,\eta,\xi) }
+M_{\delta}(t, \eta, \xi')\mu^{2}(u)
\overline{\widehat{ h}(t,\eta,\xi) }
\\
&\qquad\qquad -M_{\delta}(t, \eta, \xi')\mu^{2}(u')
\overline{\widehat{ h}(t,\eta,\xi) }
+M_{\delta}(t, \eta, \xi')\mu^{2}(u')
\overline{\widehat{ h}(t,\eta,\xi) }
\\
&\qquad\qquad\qquad\qquad -M_{\delta}(t, \eta, \xi')\mu^{2}(u')\overline{\widehat{ h}(t,\eta,\xi') }\Big)\,\,d\theta d{\eta_1} du  d\xi d\eta\\
&=A_1+A_2+A_3,
\end{align*}
with
\begin{align*}
A_1&=\int_{\mathbb{R}^4}\int_{-\frac{\pi}{2}}^{\frac{\pi}{2}} \beta(\theta)\ \langle \eta\rangle^{2r} \widehat{f}(t,{\eta_1}, u')\widehat{g}(t,\eta-{\eta_1}, \xi')\mu^{2}(u)\\
&\qquad\quad\qquad\quad\times \Big(M_{\delta}(t, \eta, \xi)
 -M_{\delta}(t, \eta, \xi')\Big)
\overline{\widehat{ h}(t,\eta,\xi) }\,d\theta d{\eta_1} du  d\xi d\eta\, ,
\end{align*}
\begin{align*}
A_2&=\int_{\mathbb{R}^4}\int_{-\frac{\pi}{2}}^{\frac{\pi}{2}} \beta(\theta)\ \langle \eta\rangle^{2r} \widehat{f}(t,{\eta_1}, u')\widehat{g}(t,\eta-{\eta_1}, \xi')M_{\delta}(t, \eta, \xi')\\ 
&\qquad\quad\qquad\quad\times \Big( \mu^{2}(u)-\mu^{2}(u')\Big)
\overline{\widehat{ h}(t,\eta,\xi) }\,d\theta d{\eta_1} du  d\xi d\eta\, ,
\end{align*}
\begin{align*}
A_3&=\int_{\mathbb{R}^4}\int_{-\frac{\pi}{2}}^{\frac{\pi}{2}} \beta(\theta)\ \langle \eta\rangle^{2r} \widehat{f}(t,{\eta_1}, u')\widehat{g}(t,\eta-{\eta_1}, \xi')M_{\delta}(t, \eta, \xi')\mu^{2}(u')\\ 
&\qquad\quad\qquad\quad\times \Big(
\overline{\widehat{ h}(t,\eta,\xi) }-\overline{\widehat{ h}(t,\eta,\xi') }\Big)\,d\theta d{\eta_1} du  d\xi d\eta\, .
\end{align*}
We study now this third terms by the following Proposition. 
\begin{proposition}
We have the following estimates:
\begin{equation*}
|A_1|\le C\|M_\delta f\|_{H^r_x(L^2_v)}\|M_\delta g\|_{H^r_x(L^2_v)}\| h\|_{H^r_x(H^s_v)},
\end{equation*}
\begin{equation*}
|A_2|\le C\|M_\delta f\|_{H^r_x(L^2_v)}|||M_\delta g|||_{(r, 0)}||| h|||_{(r, 0)},
\end{equation*}
and
\begin{equation}\label{A-3}
|A_3|\le C\|M_\delta f\|_{H^r_x(L^2_v)}\|M_\delta g\|_{H^r_x(L^2_v)}\| h\|_{H^r_x(L^2_v)}.
\end{equation}
\end{proposition}
With this Proposition, by using \eqref{2.2.1xc}, we can then get the Proposition \ref{proposition2.3},

\begin{proof} We consider first \eqref{A-3}, we can prove this it by Cancellation Lemma, in fact, 
by variable change $\xi'\to\xi$, see \cite{LX}, page 651.
\begin{align*}
    |A_3|&\leq\Big|\int_{\mathbb{R}^4}\int_{-\frac{\pi}{2}}^{\frac{\pi}{2}} \beta(\theta) M_{\delta}(t, \eta, \xi)\mu^{2}(u)
\langle \eta\rangle^r \widehat{f}(t,{\eta_1}, u)\widehat{g}(t,\eta-{\eta_1}, \xi)d{\eta_1}\\
&\qquad\qquad\times\overline{\langle \eta\rangle^r \widehat{M_\delta h}(t,\eta,\xi)} \Big\{1-\frac{1}{\cos(\theta)}\Big\}\,d\theta d{\eta_1} du  d\xi d\eta\Big|
\end{align*}
Note that \begin{align*}
    \sqrt{1+(\rho\eta+\xi)^2}&=\sqrt{1+(\rho{\eta_1}+\rho(\eta-{\eta_1})+\xi+u-u)^2}\\&\leq\sqrt{1+(\xi+\rho{\eta_1})^2}+\sqrt{1+(u+\rho(\eta-{\eta_1}))^2}+\sqrt{u}
\end{align*}
and
\begin{align*}
    \frac{e^{X+Y}}{1+\delta e^{X+Y}}\le 3\frac{e^{Y}}{1+\delta e^{Y}}\frac{e^{X}}{1+\delta e^{X}},
\end{align*}
we obtain
\begin{align*}
     M_{\delta}(t, \eta, \xi)\le 3 M_{\delta}(t, \eta_1, \xi)M_{\delta}(t, \eta-{\eta_1}, u)e^{t|u|}.
\end{align*}
Hence we have
\begin{align*}
    |A_3|&\leq\int_{\mathbb{R}^4_{\eta,\xi, u,\eta_1}}\int_{-\frac{\pi}{2}}^{\frac{\pi}{2}} \beta(\theta)\frac{\sin^2(\frac{\theta}{2})}{\cos\theta} M_{\delta}(t, \eta, \xi)\mu^{2}(u)
\{\langle \eta_1\rangle^r+\langle \eta-\eta_1\rangle^r\}\\
&\qquad\qquad\times\Big|\widehat{g}(t,{\eta_1},\xi)\widehat{f}(t,\eta-{\eta_1},u)d{\eta_1}\overline{\langle \eta\rangle^r\widehat{ h}(t,\eta,\xi) }\Big|d\theta d\xi d\eta\\
&\leq  C\|M_{\delta}f\|_{H^r_x(L^2_v)}\|M_{\delta}g\|_{L^{\infty}_x(L^2_v)}\, \| h\|_{H^r_x(L^2_v)}\\   &\qquad+C\|M_{\delta}f\|_{L^{\infty}_x(L^2_v)}\|M_{\delta}g\|_{H^r_x(L^2_v)}\, \|h\|  _{H^r_x(L^2_v)} .
\end{align*}
We finish the proof by the embedding $L^\infty(\RR_x)\subset H^{r}(\RR_x)$. Here we use 
$$
\beta(\theta)\frac{\sin^2(\frac{\theta}{2})}{\cos\theta} \ \sim\ \theta^{1-2s},
$$
is integrable over $[-\pi/2,\ \pi/2]$ for $0<s<1$.
\end{proof}

For the terms $A_1,\,\,A_2$, we have, note that $\widehat{f}$ can be divided into the odd part and the even part for $\xi$. Such as
\begin{align*}
    \widehat{f}(t,\eta,\xi)=\frac{\widehat{f}(t,\eta,\xi)+\widehat{f}(t,\eta,-\xi)}{2}+\frac{\widehat{f}(t,\eta,\xi)-\widehat{f}(t,\eta,-\xi)}{2},
\end{align*}
where\begin{align*}
     \widehat{f}_{even}(t,\eta,\xi)=\frac{\widehat{f}(t,\eta,\xi)+\widehat{f}(t,\eta,-\xi)}{2},\,\,\,\widehat{f}_{odd}(t,\eta,\xi)=\frac{\widehat{f}(t,\eta,\xi)-\widehat{f}(t,\eta,-\xi)}{2}
\end{align*}
\begin{lemma}\label{even}
 For suitable function $f$
 \begin{align*}
&\int_{\mathbb{R}^4_{\eta,\xi, u,\eta_1}}\int_{-\frac{\pi}{2}}^{\frac{\pi}{2}} \beta(\theta)\big\{\ M_{\delta}(t, \eta, \xi)-M_{\delta}(t, \eta, \xi')\big\}\mu^{2}(u)\\
&\qquad\qquad\times\langle \eta\rangle^r \widehat{g}(t,{\eta_1},\xi')\widehat{f}(t,\eta-{\eta_1},u')d{\eta_1}\overline{\langle \eta\rangle^r\widehat{h}(t,\eta,\xi) }d\theta d\xi d\eta\\
&=\int_{\mathbb{R}^4_{\eta,\xi, u,\eta_1}}\int_{-\frac{\pi}{2}}^{\frac{\pi}{2}}  \beta(\theta)\big\{\ M_{\delta}(t, \eta, \xi)-M_{\delta}(t, \eta, \xi')\big\}\mu^{2}(u)\\
&\qquad\qquad\times\langle \eta\rangle^r \widehat{g}(t,{\eta_1},\xi')\widehat{f}_{even}(t,\eta-{\eta_1},u')d{\eta_1}\overline{\langle \eta\rangle^r\widehat{h}(t,\eta,\xi) }d\theta d\xi d\eta.
\end{align*}
\end{lemma}
We can prove this Lemma by the changing of variable:  $(\xi,u,\theta)\to(\xi,-u,-\theta)$, so that  we can consider only the even part of $f$ for the terms $A_1$. 

\begin{lemma}\label{lemma5-3}
For $0<s<1$, we have
\begin{align*}
    &|A_1|\le C\|M_{\delta}f\|_{H^r_x(L^2_v)}\|M_{\delta}g\|_{H^r_x(L^2_v)}\|h\|_{H^r_x(H^s_v)}.
\end{align*}
\end{lemma}
\begin{proof}
For $A_1$, we cannot directly use the method in Lemma \ref{2.4},  because we cannot control $e^{|u'|}=e^{|\xi \sin\theta + u \cos\theta|}$ in whole space. 

We divide the space $\mathbb{R}^3_{\eta,\xi, u}\times [-\frac{\pi}{2},\ \frac{\pi}{2}]$ into four parts (depends on $\eta_1$). 
\begin{align*}
    \Theta_1&=\{|\theta|\le|\xi|^{-\frac 12};\ M_{\delta}(t,\eta-\eta_1,u')\ge M_{\delta}(t,\eta-\eta_1,-u') \}\\
    \Theta_2&=\{|\theta|\le|\xi|^{-\frac 12};\ M_{\delta}(t,\eta-\eta_1,u')\le M_{\delta}(t,\eta-\eta_1,-u') \}\\
    \Theta_3&=\{|\theta|\ge |\xi|^{-\frac 12};\ M_{\delta}(t,\eta-\eta_1,u')\ge M_{\delta}(t,\eta-\eta_1,-u')\}\\
        \Theta_4&=\{|\theta|\ge |\xi|^{-\frac 12};\ M_{\delta}(t,\eta-\eta_1,u')\le M_{\delta}(t,\eta-\eta_1,-u')\}.
\end{align*}
Then
\begin{align*}
A_1&=\sum^4_{j=1}\int_{ \Theta_j}\int_{\mathbb{R}_{\eta_1}}\beta(\theta)\ \langle \eta\rangle^{2r} \widehat{f}(t,{\eta_1}, u')\widehat{g}(t,\eta-{\eta_1}, \xi')\mu^{2}(u)\\
&\qquad\quad\qquad\quad\times \big(M_{\delta}(t, \eta, \xi)
 -M_{\delta}(t, \eta, \xi')\big)
\overline{\widehat{ h}(t,\eta,\xi) }\,d{\eta_1}d\theta du  d\xi d\eta\\
&=A_{11}+\cdots A_{14}\, .
\end{align*}
There is not the singularity of $\beta(\theta)$ for the terms $A_{13}, A_{1, 4}$, so we use Lemma \ref{even} to estimate directly,
\begin{align*}
|A_{13}|&\le \int_{ \Theta_3}\int_{\mathbb{R}_{\eta_1}}\beta(\theta)\Big|\ M_{\delta}(t, \eta, \xi)-M_{\delta}(t, \eta, \xi')\Big|\mu^{2}(u)\\
&\qquad\qquad\times\langle \eta\rangle^r |\widehat{g}(t,{\eta_1},\xi')\widehat{f}_{even}(t,\eta-{\eta_1},u')\overline{\langle \eta\rangle^r\widehat{h}(t,\eta,\xi) }|d{\eta_1}d\theta du d\xi d\eta\\
&\le\int_{ \Theta_3}\int_{\mathbb{R}_{\eta_1}}\beta(\theta)\max\Big\{\ M_{\delta}(t, \eta, \xi),M_{\delta}(t, \eta, \xi')\big\}\mu^{2}(u)\\
&\qquad\qquad\times\langle \eta\rangle^r |\widehat{g}(t,{\eta_1},\xi')\widehat{f}_{even}(t,\eta-{\eta_1},u')\overline{\langle \eta\rangle^r\widehat{h}(t,\eta,\xi) }|d{\eta_1}d\theta du d\xi d\eta.
\end{align*}
Note that
\begin{align*}
    \int_{|\theta|\ge |\xi|^{-\frac 12}}\beta(\theta)d\theta\le C \theta^{-2s}\Big|^{\pi/2}_{|\xi|^{-\frac 12}}\le C\langle\xi\rangle^{{s}}
\end{align*}
and using Lemma \ref{7num},
\begin{align*}
|A_{13}|&\le C\Big\{\int_{ \Theta_3}\beta(\theta)(\langle\xi'\rangle^{s}+\langle u'\rangle^{s})^{-1}\Big|\int_{{\eta_1}}M_{\delta}(t,{\eta_1},\xi')M_{\delta}(t,\eta-{\eta_1},u')\\
&\qquad\qquad\times\langle \eta\rangle^r \widehat{g}(t,{\eta_1},\xi')\widehat{f}_{even}(t,\eta-{\eta_1},u')d{\eta_1}\Big|^{2} du d\xi d\eta d\theta\Big\}^{1/2}\\
&\qquad\qquad\times\Big\{\int_{ \Theta_3}\beta(\theta)(\langle\xi'\rangle^{s}+\langle u'\rangle^{s})\Big|\mu^{2}(u)e^{t\langle u\rangle^{2\tilde{s}}}\langle \eta\rangle^r\widehat{h}(t,\eta,\xi)|^2  du d\xi d\eta d\theta\Big\}^{1/2}\\
&\le A_{131}\times A_{132}
\end{align*}
For the second term $A_{132}$, by using of Lemma \ref{bd},
\begin{align*}
   A_{132}&=\Big\{\int_{ \Theta_3}\beta(\theta)(\langle\xi'\rangle^{s}+\langle u'\rangle^{s})\Big|\mu^{2}(u)e^{t\langle u\rangle^s}\langle \eta\rangle^r\widehat{h}(t,\eta,\xi)\Big|^2   d\theta du d\xi  d\eta\Big\}^{1/2}\\
   &\le C\Big\{\int_{\Theta_3} \beta(\theta)(\langle\xi\rangle^{s}+\langle u\rangle^{s})\Big|\mu^{2}(u)e^{t\langle u\rangle^s}\langle \eta\rangle^r\widehat{h}(t,\eta,\xi)\Big|^2   d\theta du d\xi  d\eta\Big\}^{1/2}\\
      &\le C\Big\{\int_{\mathbb{R}^3_{\eta,\xi, u}}(\langle\xi\rangle^{2s}\langle u\rangle^{s}\Big|\mu^{2}(u)e^{t\langle u\rangle^s}\langle \eta\rangle^r\widehat{h}(t,\eta,\xi)\Big|^2  du d\xi  d\eta\Big\}^{1/2}\\
   &\le C\|h\|_{H^r_x(H^s_v)}\|\langle u\rangle^{s/2}\mu^{2}(u)e^{t\langle u\rangle^{2\tilde{s}}}\|_{L^2_u}\le C\|h\|_{H^r_x(H^s_v)}. 
\end{align*}
By variable change $(\xi',u')\to(\xi,u)$, for the first term $A_{131}$, we just need to use 
\begin{align*}
    &\langle\xi\rangle^{s'}(\langle\xi\rangle^{s}+\langle u\rangle^{s})^{-1}\le(\langle\cos\theta\xi\rangle^{s}+\langle \sin\theta u\rangle^{s})(\langle\xi\rangle^{s}+\langle u\rangle^{s})^{-1}\\
    &\ \ \ \le(\langle\xi\rangle^{s}+\langle u\rangle^{s})(\langle\xi\rangle^{s}+\langle u\rangle^{s})^{-1}\le 1,
\end{align*}
Then 
\begin{align*}
    &|A_{13}|\le C\|M_{\delta}f\|_{H^r_x(L^2_v)}\|M_{\delta}g\|_{H^r_x(L^2_v)}\|h\|_{H^r_x(H^s_v)}
\end{align*}
We similarly have
\begin{align*}
    &|A_{14}|\le C\|M_{\delta}f\|_{H^r_x(L^2_v)}\|M_{\delta}g\|_{H^r_x(L^2_v)}\|h\|_{H^r_x(H^s_v)}.
\end{align*}

Now for the term $A_{11}$ and $A_{12}$, we need to cancell the singularity of $\beta(\theta)$, two terms are  similary, so we study only the term $A_{11}$. We need now to consider two cases : mild singularity case of 
$0<s<\frac 12$, and strong singularity case of $\frac 12\le s<1$.

\noindent
{\bf Mild singularity case of $0<s<\frac 12$}

Using now the first-order Taylor formula, and $\xi'-\xi=\frac{1}{2}\sin^2(\frac{\theta}{2})\xi-\sin\theta u$
\begin{align*}
  A_{11} &=\int_{ \Theta_1}\beta(\theta)(\xi'-\xi)\mu^{2}(u)\int_{{\eta_1}}\int_0^1\partial_\xi M_{\delta}(t,\eta,\xi_\tau)d\tau\\
&\qquad\qquad\times\langle \eta\rangle^r \widehat{g}(t,{\eta_1},\xi)\widehat{f}_{even}(t,\eta-{\eta_1},u)d{\eta_1}\overline{\langle \eta\rangle^r\widehat{h}(t,\eta,\xi') }d\theta du d\xi d\eta
\\&=\int_{ \Theta_1}\beta(\theta)\frac{1}{2}\sin^2(\frac{\theta}{2})\xi\mu^{2}(u)\int_{{\eta_1}}\int_0^1\partial_\xi M_{\delta}(t,\eta,\xi_\tau)d\tau\\
&\qquad\qquad\times\langle \eta\rangle^r \widehat{g}(t,{\eta_1},\xi)\widehat{f}_{even}(t,\eta-{\eta_1},u)d{\eta_1}\overline{\langle \eta\rangle^r\widehat{h}(t,\eta,\xi') }d\theta du d\xi d\eta\\
&\qquad\qquad+\int_{ \Theta_1}\beta(\theta)\sin\theta u\mu^{2}(u)\int_{{\eta_1}}\int_0^1\partial_\xi M_{\delta}(t,\eta,\xi_\tau)d\tau\\
&\qquad\qquad\times\langle \eta\rangle^r \widehat{g}(t,{\eta_1},\xi)\widehat{f}_{even}(t,\eta-{\eta_1},u)d{\eta_1}\overline{\langle \eta\rangle^r\widehat{h}(t,\eta,\xi') }d\theta du d\xi d\eta\\
&=B_1+B_2
\end{align*}
Using the Lemma \ref{M} and Lemma \ref{7num},
\begin{align*}
    |B_1|&\le C\int_{ \Theta_1}\Big|\beta(\theta)\frac{1}{2}\sin^2(\frac{\theta}{2})\xi\mu^{2}(u)\int_{{\eta_1}}\int_0^1\partial_\xi M_{\delta}(t,\eta,\xi_\tau)d\tau\\
&\qquad\qquad\times\langle \eta\rangle^r \widehat{g}(t,{\eta_1},\xi')\widehat{f}_{even}(t,\eta-{\eta_1},u')d{\eta_1}\overline{\langle \eta\rangle^r\widehat{h}(t,\eta,\xi) }\Big|d\theta du d\xi d\eta\\
&\le C\int_{ \Theta_1}\Big|\beta(\theta)\frac{1}{2}\sin^2(\frac{\theta}{2})\xi\mu^{2}(u) M_{\delta}(t,\eta_1,\xi')M_{\delta}(t,\eta-\eta_1, u')\\
&\qquad\qquad\times e^{t\langle u\rangle^{2\tilde{s}}}\langle \eta\rangle^r \widehat{g}(t,{\eta_1},\xi')\widehat{f}_{even}(t,\eta-{\eta_1},u')d{\eta_1}\overline{\langle \eta\rangle^r\widehat{h}(t,\eta,\xi) }\Big|d\theta du d\xi d\eta.
\end{align*}
Note that , 
\begin{align*}
    \int_0^{|\xi|^{-\frac 12}}\beta(\theta)\sin^2\frac{\theta}{2}d\theta\le C_0\int_0^{|\xi|^{-\frac 12}}\theta^{1-2s}d\theta\le C_0 \theta^{2-2s}\Big|^{|\xi|^{-\frac 12}}_{0}\le C_0\langle\xi\rangle^{s-1}.
\end{align*}
Thanks to H\"older inequality
\begin{align*}
|B_1|&\le C\Big\{\int_{ \Theta_1}\beta(\theta)\sin^{2}\frac{\theta}{2}|\xi|^{1-s}\Big|\int_{{\eta_1}}M_{\delta}(t,{\eta_1},\xi')M_{\delta}(t,\eta-{\eta_1},u')\\
&\qquad \times\langle \eta\rangle^r \widehat{g}(t,{\eta_1},\xi')\widehat{f}_{even}(t,\eta-{\eta_1},u')d{\eta_1}\Big|^{2}d\theta du d\xi d\eta\Big\}^{1/2}\\
&\qquad\times\Big\{\int_{ \Theta_1}\beta(\theta)\sin^{2}\frac{\theta}{2}|\xi|^{1+s}\Big|\mu^{2}(u)e^{t\langle u\rangle^{2\tilde{s}}}\langle \eta\rangle^r\widehat{h}(t,\eta,\xi)|^2  d\theta du d\xi d\eta\Big\}^{1/2}\\
&=B_{11}\times B_{12}.
\end{align*}
For the second item $B_{12}$
\begin{align*}
   B_{12}&\le C_0\Big\{\int_{\mathbb{R}^3}\langle\xi\rangle^{2s}\Big|\mu^{2}(u)e^{t\langle u\rangle^s}\langle \eta\rangle^r\widehat{h}(t,\eta,\xi)|^2 du d\xi d\eta\Big\}^{1/2}\\
   &\le C_0 \|h\|_{H^r_x(H^s_v)}\|\mu^{2}(u)e^{t\langle u\rangle^{2\tilde{s}}}\|_{L^2_u}\le C\|h\|_{H^r_x(H^s_v)}. 
\end{align*}
By variable change $(\xi',u')\to(\xi,u)$, for the first term $B_{11}$, using
\begin{align*} \left\|\int_0^{|\xi|^{-\frac 12}}\beta(\theta)\sin^{2}\frac{\theta}{2}|\xi'|^{1-s}d\theta\right\|_{L^{\infty}_{\xi,u}}\le C.
\end{align*}
We get
\begin{align*}
    B_{11}&\le C \Big\{\int_{\mathbb{R}^3}(\langle{\eta_1}\rangle^{2r}+\langle\eta-{\eta_1}\rangle^{2r})\Big|\int_{{\eta_1}}M_{\delta}(t,{\eta_1},\xi)M_{\delta}(t,\eta-{\eta_1},u)\\
&\qquad\qquad\times\langle \eta\rangle^r \widehat{g}(t,{\eta_1},\xi)\widehat{f}_{even}(t,\eta-{\eta_1},u)d{\eta_1}\Big|^{2} du  d\xi d\eta\Big\}^{1/2}\\
&\le C \Big\{\| \langle \eta\rangle^rM_{\delta}(t,{\eta},\xi)\widehat{g}(t,{\eta},\xi)\|^2_{L^2_\eta(L^2_\xi)}\| M_{\delta}(t,{\eta},\xi)\widehat{f}_{even}(t,{\eta},\xi)\|^2_{L^1_\eta(L^2_\xi)}\\
&\qquad\qquad+\| M_{\delta}(t,{\eta},\xi)\widehat{g}(t,{\eta},\xi)\|^2_{L^1_\eta(L^2_\xi)}\| \langle \eta\rangle^rM_{\delta}(t,{\eta},\xi)\widehat{f}_{even}(t,{\eta},\xi)\|^2_{L^2_\eta(L^2_\xi)}\Big\}^{1/2}\\
&\le C\Big\{\|M_{\delta}g\|^2_{H^r_x(H^s_v)}\|M_{\delta}f\|^2_{L^{\infty}_x(L^2_v)}+\|M_{\delta}g\|^2_{L^{\infty}_x(H^s_v)}\|M_{\delta}f\|^2_{H^r_x(L^2_v)}\Big\}^{1/2},
\end{align*}
Then note that continuous embedding $L^\infty(\RR)\subset H^r(\RR)$, 
\begin{align*}
    &|B_1|\le C\|M_{\delta}f\|_{H^r_x(L^2_v)}\|M_{\delta}g\|_{H^r_x(L^2_v)}\|h\|_{H^r_x(H^s_v)}.
\end{align*}

For $B_2$, by the same method, using the Lemma \ref{M} and Lemma \ref{7num},
\begin{align*}
    |B_2|&\le C\int_{ \Theta_1}|\beta(\theta)\frac{1}{2}\sin(\theta)u\mu^{2}(u)\int_{{\eta_1}}\int_0^1\partial_\xi M_{\delta}(t,\eta,\xi_\tau)d\tau\\
&\qquad\qquad\times\langle \eta\rangle^r \widehat{g}(t,{\eta_1},\xi')\widehat{f}_{even}(t,\eta-{\eta_1},u')d{\eta_1}\overline{\langle \eta\rangle^r\widehat{h}(t,\eta,\xi) }|d\theta du d\xi d\eta\\
&\le C\int_{ \Theta_1}|\beta(\theta)\frac{1}{2}\sin(\theta)u\mu^{2}(u)M_{\delta}(t,\eta_1,\xi')M_{\delta}(t,\eta-\eta_1,\xi')\\
&\qquad\times e^{t\langle u\rangle^{2\tilde{s}}}\langle \eta\rangle^r \widehat{g}(t,{\eta_1},\xi')\widehat{f}_{even}(t,\eta-{\eta_1},u')d{\eta_1}\overline{\langle \eta\rangle^r\widehat{h}(t,\eta,\xi) }|d\theta du d\xi d\eta
\end{align*}
Using, for $0<2s<1$, 
\begin{align*}   \int_0^{|\xi|^{-\frac 12}}\beta(\theta)\sin(\theta)d\theta\le C_0\int_0^{|\xi|^{-\frac 12}}\theta^{-2s}d\theta\le C_0 \theta^{1-2s}\Big|^{|\xi|^{-\frac 12}}_{0}\le C_0\langle\xi\rangle^{s-\frac 12}.
\end{align*}
We get
\begin{align*}
    &|B_2|\le C\|M_{\delta}f\|_{H^r_x(L^2_v)}\|M_{\delta}g\|_{H^r_x(L^2_v)}\|h\|_{H^r_x(L^2_v)}.
\end{align*}
Hence we have proved Lemma \ref{lemma5-3} for $0<s<\frac 12$, 
$$
|A_1|\le C\|M_{\delta}f\|_{H^r_x(L^2_v)}\|M_{\delta}g\|_{H^r_x(L^2_v)}\|h\|_{H^r_x(H^s_v)}.
$$
\end{proof}

\noindent
{\bf Strong singularity case of $\frac 12\le s<1$.}

The analysis in $\Theta_3,\,\,\Theta_4$ is same to the case when $0<s<1/2$. For $\Theta_1,\,\,\Theta_2$, take the second-order Taylor formula. Using
\[
\iint\beta(\theta)(\xi-\xi')F(\xi')d\theta d\xi=0,
\]
and
$$
(\xi'-\xi)^2=\frac{1}{4}\sin^4(\frac{\theta}{2})\xi^2+\sin^2\theta u^2+\sin^2(\frac{\theta}{2})\sin\theta\, \xi\, u,
$$
Thus, combined with Lemma \ref{M} and Lemma \ref{even}, we have
\begin{align*}
A_1&=\int_{\Theta_1}\int_{{\eta_1}}\beta(\theta)\big\{\ M_{\delta}(t, \eta, \xi)-M_{\delta}(t, \eta, \xi')\big\}\mu^{2}(u)\\
&\ \ \times\langle \eta\rangle^r \widehat{g}(t,{\eta_1},\xi')\widehat{f}_{even}(t,\eta-{\eta_1},u')d{\eta_1}\overline{\langle \eta\rangle^r\widehat{h}(t,\eta,\xi) }d\theta du  d\xi d\eta\\
&=\int_{\Theta_1}\int_{{\eta_1}}\beta(\theta)\int_0^1(\xi-\xi')^2(1-\tau)\partial^2_\xi M_{\delta}(t,\eta,\xi_\tau)d\tau \mu^{2}(u)\\
&\ \ \times\langle \eta\rangle^r \widehat{g}(t,{\eta_1},\xi')\widehat{f}_{even}(t,\eta-{\eta_1},u')d{\eta_1}\overline{\langle \eta\rangle^r\widehat{h}(t,\eta,\xi) }d\theta du  d\xi d\eta\\
&=B_3+B_4+B_5
\end{align*}
where
\begin{align*}
    B_3&=\int_{\Theta_1}\int_{{\eta_1}}\beta(\theta)\int_0^1\sin^4(\frac{\theta}{2})\xi^2(1-\tau)\partial^2_\xi M_{\delta}(t,\eta,\xi_\tau)d\tau\mu^{2}(u)\\
&\qquad\qquad\times\langle \eta\rangle^r \widehat{g}(t,{\eta_1},\xi')\widehat{f}_{even}(t,\eta-{\eta_1},u')d{\eta_1}\overline{\langle \eta\rangle^r\widehat{h}(t,\eta,\xi) }d\theta d\xi du  d\eta,
\end{align*}
\begin{align*}
B_4&=\int_{\Theta_1}\int_{{\eta_1}}\beta(\theta)\int_0^1\sin^2\theta\, u^2(1-\tau)\partial^2_\xi M_{\delta}(t,\eta,\xi_\tau)d\tau\mu^{2}(u)\\
&\qquad\qquad\times\langle \eta\rangle^r \widehat{g}(t,{\eta_1},\xi')\widehat{f}_{even}(t,\eta-{\eta_1},u')d{\eta_1}\overline{\langle \eta\rangle^r\widehat{h}(t,\eta,\xi) }d\theta  du d\xi d\eta,
\end{align*}
\begin{align*}
    B_5&=\int_{\Theta_1}\int_{{\eta_1}}\beta(\theta)\int_0^1\sin^2(\frac{\theta}{2})\sin\theta\xi u(1-\tau)\partial^2_\xi M_{\delta}(t,\eta,\xi_\tau)d\tau\mu^{2}(u)\\
&\qquad\qquad\times\langle \eta\rangle^r \widehat{g}(t,{\eta_1},\xi')\widehat{f}_{even}(t,\eta-{\eta_1},u')d{\eta_1}\overline{\langle \eta\rangle^r\widehat{h}(t,\eta,\xi) }d\theta du  d\xi d\eta.
\end{align*}
We can analysis $B_3,B_4,B_5$ by same method, for instance we estimate a term of $B_3$.

Note that when $(\xi, u, \theta)\in\Theta_1$, $
|\xi|\sin^2(\frac{\theta}{2})\le 1$ and
\begin{align*}
    \int_0^{|\xi|^{-\frac 12}}\beta(\theta)\sin^2\frac{\theta}{2}d\theta\le C_0\int_0^{|\xi|^{-\frac 12}}\theta^{1-2s}d\theta\le C_0 \theta^{2-2s}\Big|^{|\xi|^{-\frac 12}}_{0}\le C_0\langle\xi\rangle^{s-1}.
\end{align*}
We can get
\begin{align*}
    &|B_3|\le C\|M_{\delta}g\|_{H^r_x(L^2_v)}\|M_{\delta}f\|_{H^r_x(L^2_v)}\|h\|_{H^r_x(H^{s}_v)}.
\end{align*}
By the same method, 
\begin{align*}
  |B_{4}|
&\le C\Big\{\int_{\Theta_1}\beta(\theta)\sin^{2}{\theta}\ \Big|\int_{{\eta_1}}M_{\delta}(t,{\eta_1},\xi')M_{\delta}(t,\eta-{\eta_1},u')\\
&\qquad \times\langle \eta\rangle^r \widehat{g}(t,{\eta_1},\xi')\widehat{f}_{even}(t,\eta-{\eta_1},u')d{\eta_1}\Big|^{2}d\theta du  d\xi d\eta\Big\}^{1/2}\\
&\qquad\times\Big\{\int_{\Theta_1}\beta(\theta)\sin^{2}{\theta}u^4\Big|\mu^{2}(u)e^{t\langle u\rangle^{2\tilde{s}}}\langle \eta\rangle^r\widehat{h}(t,\eta,\xi)\Big|^2  d\theta du  d\xi d\eta\Big\}^{1/2},
\end{align*}
using
\begin{align*}
    \int_0^{|\xi|^{-\frac 12}}\beta(\theta)\sin^{2}{\theta}d\theta\le C\langle\xi\rangle^{s-1}\le C.
\end{align*}
Which imply
$$
|B_4|\le C\|M_{\delta}g\|_{H^r_x(L^2_v)}\|M_{\delta}f\|_{H^r_x(L^2_v)}\|h\|_{H^r_x(L^{2}_v)}.
$$
And for $B_5$
\begin{align*}
  |B_{5}|
&\le C\Big\{\int_{\Theta_1}\beta(\theta)\sin^{2}{\theta}|\xi|^{1-s}\Big|\int_{{\eta_1}}M_{\delta}(t,{\eta_1},\xi')M_{\delta}(t,\eta-{\eta_1},u')\\
&\qquad \times\langle \eta\rangle^r \widehat{g}(t,{\eta_1},\xi')\widehat{f}_{even}(t,\eta-{\eta_1},u')d{\eta_1}\Big|^{2}d\theta du  d\xi d\eta\Big\}^{1/2}\\
&\times\Big\{\int_{\Theta_1}\beta(\theta)\sin^{4}\frac{\theta}{2}|\xi|^{1+s}u^2\Big|\mu^{2}(u)e^{t\langle u\rangle^{2\tilde{s}}}\langle \eta\rangle^r\widehat{h}(t,\eta,\xi)|^2  d\theta du  d\xi d\eta\Big\}^{1/2}.
\end{align*}
Now using 
\begin{align*}
\left\|\int_0^{|\xi|^{-\frac 12}}\beta(\theta)\sin^{2}{\theta}|\xi'|^{1-s}d\theta\right\|_{L^{\infty}_{\xi,u}}&= \left\|\int_0^{|\xi|^{-\frac 12}}\beta(\theta)\sin^{2}{\theta}|\xi|^{1-s}d\theta\right\|_{L^{\infty}_{\xi,u}}\\
   & \le C_0 |||\xi|^{1-s}\langle\xi\rangle^{s-1}\|_{L^{\infty}_{\xi,u}}\le C_0,
\end{align*}
and
\begin{align*}
\left\|\int_0^{|\xi|^{-\frac 12}}\beta(\theta)\sin^{4}{\frac{\theta}{2}}|\xi|^{1+s}d\theta\right\|_{L^{\infty}_{\xi,u}} \le C_0 |||\xi|^{1+s}\langle\xi\rangle^{s-2}\|_{L^{\infty}_{\xi,u}}\le C_0.
\end{align*}
Thus
\begin{align*}
    &|B_5|\le C\|M_{\delta}f\|_{H^r_x(L^2_v)}\|M_{\delta}g\|_{H^r_x(L^2_v)}\|h\|_{H^r_x(L^2_v)}.
\end{align*}
Hence we have proved Lemma \ref{lemma5-3} for $\frac 12\le s<1$.

\section{Technical Lemma}

The estimate of term $A_2$ is more technical, we have firstly.
\begin{corollary}\label{Lemma-A2}
    For $A_2$, note that $\mu$ is a even function for $v$, we also have
    \begin{align*}
A_2&=\int_{\mathbb{R}^4}\int_{-\frac{\pi}{2}}^{\frac{\pi}{2}} \beta(\theta)\ \langle \eta\rangle^{2r} \widehat{f}_{even}(t,{\eta_1}, u')\widehat{g}(t,\eta-{\eta_1}, \xi')M_{\delta}(t, \eta, \xi')\\ 
&\qquad\quad\times \big( \mu^{2}(u)-\mu^{2}(u')\big)
\overline{\widehat{h}(t,\eta,\xi)}\,d\theta d{\eta_1} du  d\xi d\eta\, ,
\end{align*}
\end{corollary}
We can prove now,
\begin{lemma}
We have the following estimates
 \begin{equation*}
    |A_2|\le C\|M_{\delta}f\|_{H^r_x(L^2_v)}|||M_{\delta}g|||_{(r,0)}|||h|||_{(r,0)}
\end{equation*}
\end{lemma}
\begin{proof}
For $A_2$, we can analysis it by the same way like $A_1$. 
But we divide the space $\mathbb{R}^3_{\eta,\xi, u}\times [-\frac{\pi}{2},\ \frac{\pi}{2}]$ into 6 parts,
\begin{align*}
\Sigma_1&=\{|\theta|\le |\xi|^{-1},\mu(u)\le\mu(u') \},\ \ \ \ \ 
    \Sigma_2=\{|\theta|\ge  |\xi|^{-1},\mu(u)\le\mu(u') \},\\
   \Sigma_3&=\{|\theta|\ge  |\xi|^{-1}, M_{\delta}(t,\eta-\eta_1,u')\ge M_{\delta}(t,\eta-\eta_1,-u')\,,\mu(u)\ge\mu(u')\},\\
       \Sigma_4&=\{|\theta|\ge  |\xi|^{-1},M_{\delta}(t,\eta-\eta_1,u')\le M_{\delta}(t,\eta-\eta_1,-u')\,,\mu(u)\ge\mu(u')\},
    \\
        \Sigma_5&=\{|\theta|\le  |\xi|^{-1},M_{\delta}(t,\eta-\eta_1,u')\ge M_{\delta}(t,\eta-\eta_1,-u')\,,\mu(u)\ge\mu(u')\},\\
        \Sigma_6&=\{|\theta|\le  |\xi|^{-1},M_{\delta}(t,\eta-\eta_1,u')\le M_{\delta}(t,\eta-\eta_1,-u')\,,\mu(u)\ge\mu(u')\}.
\end{align*}
Using Corollary \ref{Lemma-A2}, we have
\begin{align*}
A_2&=\sum^{6}_{j=1}\int_{\Sigma_j}\int_{{\eta_1}}\beta(\theta)M_{\delta}(t, \eta, \xi')\Big( \mu^{2}(u)-\mu^{2}(u')\Big)\\
&\qquad\times\langle \eta\rangle^r \widehat{g}(t,{\eta_1},\xi')\widehat{f}_{even}(t,\eta-{\eta_1},u')d{\eta_1}\overline{\langle \eta\rangle^r\widehat{h}(t,\eta,\xi) }d\theta du d\xi d\eta\\
&=A_{21}+\cdots A_{26}.
\end{align*}
Firstly
\begin{align*}
|A_{22}|&\le\int_{\Sigma_2}\beta(\theta)\int_{{\eta_1}}M_{\delta}(t, \eta, \xi')\mu^{2}(u')\\
&\qquad\times\langle \eta\rangle^r |\widehat{g}(t,{\eta_1},\xi')\widehat{f}_{even}(t,\eta-{\eta_1},u')d{\eta_1}\overline{\langle \eta\rangle^r\widehat{h}(t,\eta,\xi) }|d\theta d\xi d\eta.
\end{align*}
Note that 
\begin{align*}
    &\sqrt{1+(\rho\eta+\xi')^2}=\sqrt{1+(\rho{\eta_1}+\rho(\eta-{\eta_1})+\xi+u'-u')^2}\\&\leq\sqrt{1+(\xi'+\rho{\eta_1})^2}+\sqrt{1+(u'+\rho(\eta-{\eta_1}))^2}+\sqrt{|u'|}.
\end{align*}
It stands 
\begin{align*}
    \frac{M_{\delta}(t,\eta,\xi')}{M_{\delta}(t,\eta_1,\xi')M_{\delta}(t,\eta-\eta_1,u')}\le 3e^{|u'|}.
\end{align*}
We have also
\begin{align*}
    \int_{|\theta|\ge  |\xi|^{-1}}\beta(\theta)d\theta\le C \theta^{-2s}\Big|^{\pi/2}_{ |\xi|^{-1}}\le C\langle\xi\rangle^{{2s}}.
\end{align*}
We can then get
\begin{equation*}
    |A_{22}|\le C\|M_{\delta}f\|_{H^r_x(L^2_v)}\|M_{\delta}g\|_{H^r_x(H^s_v)}\|h\|_{H^r_x(H^s_v)}.
\end{equation*}
The estimate of the term $A_{23}, A_{24}$, the integrale define over $\Sigma_3$ and $\Sigma_4$, are exactly the same 
\begin{equation*}
    |A_{23}|+|A_{24}|\le C\|M_{\delta}f\|_{H^r_x(L^2_v)}\|M_{\delta}g\|_{H^r_x(H^s_v)}\|h\|_{H^r_x(H^s_v)}.
\end{equation*}
The estimate of the term $A_{21}, A_{25}, A_{26}$, the integrale define over $\Sigma_1, \Sigma_5$ and $\Sigma_6$, are also by the same way, so we study only the term $A_{21}$. 

When $0<s<1/2$, using the first-order Taylor formula 
\[
\mu^2(u)-\mu^2(u')=\int_{0}^{1}(\mu^2)'(u_\tau)\,(u-u')\,d\tau.
\]
where $u_\tau=u-\tau(u-u')=u-(1-\tau)(u-u')$ and $u'-u=\frac{1}{2}\sin^2(\frac{\theta}{2})u+\sin\theta \xi$.
\begin{align*}
A_{21}&=\int_{\Sigma_1}\int_{{\eta_1}}\beta(\theta)M_{\delta}(t, \eta, \xi')\int_{0}^{1}(\mu^2)'(u_\tau )\,(u-u')d\tau \\
&\qquad\qquad\times\langle \eta\rangle^r \widehat{g}(t,{\eta_1},\xi')\widehat{f}_{even}(t,\eta-{\eta_1},u')\overline{\langle \eta\rangle^r\widehat{h}(t,\eta,\xi) }d{\eta_1}d\theta du d\xi d\eta.
\end{align*}
Then
\begin{align*}
A_{21}
&=\int_{\Sigma_1}\int_{{\eta_1}}\beta(\theta)\frac{1}{2}\sin^2(\frac{\theta}{2})\,u\,M_{\delta}(t, \eta, \xi')\int_{0}^{1}(\mu^2)'(u_\tau)d\tau\\
&\qquad\qquad\times\langle \eta\rangle^r \widehat{g}(t,{\eta_1},\xi)\widehat{f}_{even}(t,\eta-{\eta_1},u)\overline{\langle \eta\rangle^r\widehat{h}(t,\eta,\xi') }d{\eta_1} d\theta du d\xi d\eta\\
&\qquad\qquad+\int_{\Sigma_1}\int_{{\eta_1}}\beta(\theta)\sin\theta\, \xi \, M_{\delta}(t, \eta, \xi')\int_{0}^{1}(\mu^2)'(u_\tau)d\tau\\
&\qquad\qquad\times\langle \eta\rangle^r \widehat{g}(t,{\eta_1},\xi)\widehat{f}_{even}(t,\eta-{\eta_1},u)d{\eta_1}\overline{\langle \eta\rangle^r\widehat{h}(t,\eta,\xi') }d\theta du d\xi d\eta\\
&=D_1+D_2
\end{align*}
Since for all $ (\xi,u,\eta,\theta)\in \Sigma_1,\,\, |(\sin\theta)\xi|\le 1$,
 $|\frac{1}{2}u|-1\le|u_{\tau}|\le |u+1|$, $|\frac{1}{2}u-1|\le|u'|\le |u|+1$ and $(\mu^2)'(u)=-2u\mu^2(u)$, we get
\begin{align*}
    |D_1|&\le C\int_{\Sigma_1}\int_{{\eta_1}}\Big|\beta(\theta)\frac{1}{2}\sin^2(\frac{\theta}{2})u M_{\delta}(t,\eta,\xi')(|u|+1)\mu^{2}(|\frac{1}{2}u-1|)\\
&\qquad\qquad\times\int_{{\eta_1}}\langle \eta\rangle^r \widehat{g}(t,{\eta_1},\xi')\widehat{f}_{even}(t,\eta-{\eta_1},u')d{\eta_1}\overline{\langle \eta\rangle^r\widehat{h}(t,\eta,\xi) }\Big|d\theta d\xi d\eta.
\end{align*}
Note that when $(\xi,u, \eta, \theta)\in\Sigma_1$, 
\begin{align*}
    \frac{M_{\delta}(t,\eta,\xi')}{M_{\delta}(t,\eta_1,\xi')M_{\delta}(t,\eta-\eta_1,u')}\le 3e^{|u|+1},
\end{align*}
and
\begin{align*}
    \int_0^{|\xi|^{-1}}\beta(\theta)\sin^2\frac{\theta}{2}d\theta\le C_0\langle\xi\rangle^{2s-2}
\end{align*}
Thanks to H\"older inequality, we get
$$ 
|D_1|\le C\|M_{\delta}f\|_{H^r_x(L^2_v)}\|M_{\delta}g\|_{H^r_x(H^s_v)}\|h\|_{H^r_x(H^s_v)}.
$$
For $D_2$, we have
\begin{align*}
    |D_2|&\le C\Big\{\int_{\Sigma_1}\int_{{\eta_1}}\beta(\theta)\sin(\theta)|\xi'|\Big|M_{\delta}(t,{\eta_1},\xi')M_{\delta}(t,\eta-{\eta_1},u')\\
&\qquad \times\langle \eta\rangle^r \widehat{g}(t,{\eta_1},\xi')\widehat{f}_{even}(t,\eta-{\eta_1},u')d{\eta_1}\Big|^{2}d\theta du d\xi d\eta\Big\}^{1/2}\\
&\qquad\times\Big\{\int_{\Sigma_1}\int_{{\eta_1}}\beta(\theta)\sin(\theta)|\xi|^2|\xi'|^{-1}\Big|(|u|+1)\mu^{2}(|\frac{1}{2}u-1|)\\
&\qquad\qquad\qquad\qquad \times e^{|u|+1}\langle \eta\rangle^r\widehat{h}(t,\eta,\xi)\Big|^2  d\theta d\xi d\eta\Big\}^{1/2}.
\end{align*}
By the same method, we use, for $0<s<\frac 12$,
\begin{align*}
    \int_0^{|\xi|^{-1}}\beta(\theta)\sin(\theta)d\theta\le C_0\langle\xi\rangle^{2s-1}
\end{align*}
and
\begin{align*}
    |\xi|^2/|\xi'|\le \langle\xi\rangle\langle u\rangle,\qquad\forall\ (\xi,u,\eta,\theta)\in\Sigma_1.
\end{align*}
We get
\begin{align*}
    &|D_2|\le C\|M_{\delta}f\|_{H^r_x(L^2_v)}\|M_{\delta}g\|_{H^r_x(H^s_v)}\|h\|_{H^r_x(H^s_v)}.
\end{align*}
Hence, we have for $0<s<1/2$, 
\begin{equation*}
|A_{21}|+|A_{25}|+|A_{26}|\le C\|M_{\delta}f\|_{H^r_x(L^2_v)}\|M_{\delta}g\|_{H^r_x(H^s_v)}\|h\|_{H^r_x(H^s_v)}.
\end{equation*}
For the case of $1/2\le s<1$, we use second order Tylor fromula,
\[
\mu^2(u)-\mu^2(u')=-2u'\mu^2(u')(u-u')
+\int_{0}^{1}(1-\tau)\,(\mu^2){"}(u_\tau)\,(u-u')^{2}\,d\tau.
\]
And when $(\xi,u, \eta,\theta)\in\Sigma_1$, 
$$
(u'-u)^2=\frac{1}{4}\sin^4(\frac{\theta}{2})u^2+\sin\theta^2 \xi^2-\sin^2(\frac{\theta}{2})\sin\theta \, \xi\, u,
$$
and
\begin{align*}
    \int_0^{|\xi|^{-1}}\beta(\theta)\sin^2\frac{\theta}{2}d\theta\le C_0\langle\xi\rangle^{2s-2}.
\end{align*}
It can therefore be concluded that
\begin{align*}
|A_2| \leq  C||M_{\delta}f||_{H^r_x(L^2_v)}|||M_{\delta}g|||_{(r,0)}|||h|||_{(r,0)}.
\end{align*}
We finish then the proof of  Proposition \ref{proposition2.3}.
\end{proof}

\vskip0.5cm
\section{Exponential decay for the velocity variable}

To study the exponential decay for the velocity variable of the solution, we choose test function to the equation in \eqref{eq-1}, then we will choose
$$
G^2_{\delta} g(t)\in L^\infty(]0, T[; \mathcal{G}^{1+\frac{1}{2s}}(\mathbb{R}_x; S^{1+\frac{1}{2s}}_{1+\frac{1}{2s}}(\mathbb{R}_v))),
$$
where
$$
G_\delta(t, v)=\frac{e^{c_0t \langle v\rangle^{2\tilde{s}}}}{1+\delta e^{c_0t \langle v\rangle^{2\tilde{s}}}},
$$
is a bounded function for $0<\delta< 1$, then we have
\begin{equation}\label{2.6ab}
    \begin{aligned}
        \Big((\partial_t +v\partial_x)g, &G_{\delta}^2g \Big)_{H^r_x(L^2_v)}+\Big(\mathcal{L}g, G_{\delta}^2g \Big)_{H^r_x(L^2_v)}\\
        &\qquad \qquad =\Big(\mathcal{K}(g, g),  G_{\delta}^2g \Big)_{H^r_x(L^2_v)}.
    \end{aligned}
\end{equation}
Our object is also, by using the above equation, to establish the following estimate
\begin{equation}\label{2.7ab}
    \|G_{\delta} g(t) \|_{H^r_x(L^2_v)}\le C,\ \ \ \ \ 
    0\le t\le T,
\end{equation}
with $C$ independs of $0<\delta<1$, then take $\delta\,\to\, 0$ in \eqref{2.7ab} to get 
$$
\|e^{c_0t \langle\ \cdot\ \rangle^{2\tilde{s}}} g(t) \|_{H^r_x(L^2_v)}\le C,\ \ \ \ \ 
    0\le t\le T.
$$
Combine with the results of \eqref{2.1}, we get the Gevrey-Gelfand-Shilov smoothing effect of Theorem \ref{theo1}.

We study the three terms in \eqref{2.6ab} by the following  Proposition,
\begin{proposition} \label{proposition2.1ab}
For the kinetic part of  \eqref{2.6ab}, we have 
\begin{equation*}
    \Big((\partial_t +v\partial_x)g, G_{\delta}^2g \Big)_{H^r_x(L^2_v)}\ge \frac{1}{2}\frac{d}{dt}||G_{\delta}g||^2_{H^r_x(L^2_v)}-c_0||\langle v\rangle^s G_{\delta}g||^2_{H^r_x(L^2_v)}.  
\end{equation*}
For the linaer part, we have
\begin{equation*}
    \big(\mathcal{L}g, G_{\delta}^2g \big)_{H^r_x(L^2_v)}\ge c_1 ||| G_{\delta} g|||_{(r,0)}^2 -C_1||G_{\delta}g||^2_{H^r_x(L^2_v)}.  
\end{equation*} 
For the nonlinear part, 
\begin{equation*}
\begin{aligned}
    \left|\big(\mathcal{K}(g, g), G_{\delta}^2g \big)_{H^r_x(L^2_v)}\right|&\le C_2 
    \|g\|_{H^r_x(L^2_v)}||| G_{\delta} g|||_{(r,0)}^2 \\
    &\quad +C_3||G_{\delta}g||^2_{H^r_x(L^2_v)}\ ||| G_{\delta} g|||_{(r,0)}. 
\end{aligned}
\end{equation*}
\end{proposition}
Then the nonlinear Gronwell inequality give immediatly the following results.
\begin{equation*}
||G_\delta g(t)||^2_{H^r_x(L^2_v)}+\int^t_0||||G_\delta g(\tau)|||^2_{(r,0)}d\tau\leq \frac{\|g_0\|^2_{H^r_x(L^2_v)}}{1-Ct\|g_0\|^2_{H^r_x(L^2_v)}}.
\end{equation*}    
Exactly as in Section \ref{section3}, we only need to study the estimates of commutators.

To prove the commutator estimation between linear operator $\mathcal{L}$ and $G_\delta$, we give the following lemma.

\begin{lemma}\label{5.1}
For $G_\delta(t,v)$, we have
\begin{align*}
|\partial_v G_{\delta}(t,v)|&\le C \langle v\rangle^{2\tilde{s}-1}G_{\delta}(t,v),\\
|\partial^2_v G_{\delta}(t,v)|&\le C \langle v\rangle^{4\tilde{s}-2}G_{\delta}(t,v).
\end{align*}
\end{lemma}

Then, we deal with the commutator estimation of linear operator $\mathcal{L}$ and $G_\delta$ as follows.

\begin{lemma}\label{pr3}
Assume that the cross-section satisfies \eqref{1.1+1} with $0<s<1,\,0<\delta<1$ and $r>\frac{1}{2}$. Then
\begin{equation*}
    |([\mathcal{L},\ G_\delta]g, h)_{H^r_x(L^2_v)}|\le C\|\langle v\rangle^s G_\delta g\|_{H^r_x(L^2_v)}\|h\|_{H^r_x(L^2_v)},
\end{equation*}
\end{lemma}
The proof is almost the same as the Fourier multiplier, and for the case of $0<s<\frac 12$, it follows from the first-order Taylor formula,
\begin{align*}
    G_\delta(t, v)-G_\delta(t, v’)=\int_0^1(v-v')\partial_vG_\delta(t, v_\tau)d\tau,
\end{align*}
and
$$
(v-v')=2\sin^2\frac{\theta}{2}\, v-\sin\theta\,  v_*=\tan\frac{\theta}{2}(v_*+v_*'),
$$
we need also 
$\langle v_{\tau}\rangle^{2\tilde{s}}\le \langle v'\rangle^{2\tilde{s}}+|v_*|+|v_*'|$.

For the Case of $\frac 12\le s<1$, we use the second-order Taylor formula
\begin{align*}
    G_\delta(t, v)-G_\delta(t, v’)=(v-v')&\partial_vG_\delta(t, v')\\
    &+\int_0^1(1-\tau)(v-v')^2\partial_v^2G_\delta(t, v_\tau)d\tau,
\end{align*}
We give here some computation for non linear commutaors, and the linear case can be deduce immediatly.
\begin{lemma}\label{unlinearD}
Assume that the cross-section satisfies \eqref{1.1+1} with $0<s<1,\,\delta<1$ and $r>1/2$. Then we have
\begin{align*}
\Big|\Big( &\mathcal{K}(f,\ G_{\delta}g)-G_{\delta}\mathcal{K}(f,\ g),\  h \Big)_{H^r_x(L^2_v)}\Big| \\
&\leq C ||G_{\delta}{f}\|_{H^r_x(L^2_v)}\| G_{\delta}g\|_{H^r_x(L^2_v)}|||h|||_{(r,0)}.
\end{align*}
\end{lemma}
\begin{proof}
We have firstly
\begin{align*}
\Big( &\mathcal{K}(f,\ G_{\delta}g)-G_{\delta}\mathcal{K}(f\ ,g),\  h \Big)_{H^r_x(L^2_v)} \\
&=
\int_{\mathbb{R}^3}\int^{\frac{\pi}{2}}_{-\frac{\pi}{2}}\mu^{-\frac{1}{2}}(v_*)\beta(\theta)\Big\{ G_\delta(t, v')\langle D_x\rangle^rg(t,x,v')f(t,x,v'_*)\\
&\qquad-\langle D_x\rangle^rG_\delta(t, v)g(t,x,v)f(t,x,v_*)\Big\}\langle D_x\rangle^r\bar{h}(t,x,v) d\theta dv_* dv dx\\
&=\int_{\mathbb{R}^3}\int^{\frac{\pi}{2}}_{-\frac{\pi}{2}}\mu^{-\frac{1}{2}}(v)\beta(\theta) \mu(v_*')\mu^{\frac{1}{2}}(v')\{G_\delta(t, v')-G_\delta(t, v)\}\\
&\qquad\qquad\times \langle D_x\rangle^rg(t, x, v')f(t,x,v'_*) \langle D_x\rangle^rh(t,x,v) d\theta dv_* dv dx.
\end{align*}

\textbf{The case of $0<s<1/2$: }
Note that $|v|^2=|v'|^2+|v'_* |^2-|v_*|^2$, thus when $0<s<\frac 12$
\begin{align*}
    \langle v_\tau\rangle^s\le \langle v'\rangle^s+\langle v_*'\rangle^s+C\langle v_*\rangle^s,
\end{align*}
we have
\begin{equation*}
    G_\delta(t, v_\tau)\le 3G_\delta(t, v')G_\delta(t, v_*')e^{c_0t\langle v_*\rangle^s}.
\end{equation*}
We also divide the space into two parts,
\begin{align*}
    \Xi_1=\Big\{|\theta|\le\frac{1}{|v|^s}\Big\},\qquad\qquad\Xi_2=\Big\{|\theta|\ge \frac{1}{|v|^s}\Big\}.
\end{align*}
Then
\begin{align*}
\Big( &\mathcal{K}(f,\ G_{\delta}g)-G_{\delta}\mathcal{K}(f,\ g),\  h \Big)_{H^r_x(L^2_v)}\\
&=\int_{\mathbb{R}^3}\int^{\frac{\pi}{2}}_{-\frac{\pi}{2}}\beta(\theta)\mu^{\frac{1}{2}}(v_*)\int_0^1\frac{1}{2}\sin^2\!\left(\frac{\theta}{2}\right)v\partial_vG_\delta(t, v_\tau)d\tau\\
&\qquad\qquad\times \langle D_x\rangle^rg(t, x, v')f(t,x,v_*') \langle D_x\rangle^rh(t,x,v) d\theta dv_* dv dx\\
&\qquad+\int_x\int_v\int_{v_*}\int_{\theta}\beta(\theta)\mu^{\frac{1}{2}}(v_*)\int_0^1\sin\theta\,v_*\partial_vG_\delta(t, v_\tau)d\tau\\
&\qquad\qquad\times \langle D_x\rangle^rg(t, x, v')f(t,x,v_*') \langle D_x\rangle^rh(t,x,v) d\theta dv_* dv dx\\
&=E_1+E_2.
\end{align*}
Thus, by Lemma \ref{5.1}, for the integrale over $\Xi_1$ 
\begin{align*}
|E_{11}|
&\le C\Big\{\int_{\mathbb{R}^2_{x,v_*}}\int_{\Xi_1}\beta(\theta)\sin\theta\Big|G_\delta(t, v')G_\delta(t, v_*')\\
&\qquad\qquad\times\langle D_x\rangle^rg(t, x, v')f(t,x,v_*') \Big|^2 d\theta dv_* dv dx\Big\}^{1/2}\\
&\qquad\times\Big\{\int_x\int_v\int_{v_*}\int_{\theta}\beta(\theta)\sin\theta\Big|\mu^{\frac{1}{2}}(v_*)\langle v_*\rangle^{1-2\tilde{s}}e^{c_0t |v_*|^{s}}\\
&\qquad\qquad\times\langle D_x\rangle^rh(t,x,v)\langle v \rangle^{s} \Big|^2 d\theta dv_* dv dx\Big\}^{1/2},
\end{align*}
which imply 
That means,
\begin{align*}
    |E_{11}|\le C \|G_\delta g\|_{H^r_x(L^2_v)}\|G_\delta f\|_{H^r_x(L^2_v)}\|\langle v \rangle^{s}h\|_{H^r_x(L^2_v)},
\end{align*}
and the same for $E_{21}$. We omite the estimates of integrale over $\Xi_2$, and also the case of $\frac 12\le s<1$.
\end{proof}

\bigskip
\noindent {\bf Acknowledgements.} 
The second author (H. M. Cao) was supported by the NSFC (No. 12001269) and the Fundamental Research Funds for the Central Universities of China. 
The third author (C.-J. Xu) was supported by the NSFC (No.12031006, No.12426632) and the Fundamental Research Funds for the Central Universities of China.

\bigskip
\noindent
{\bf Data availability.} No data was used for the research described in the article.

\bigskip\noindent
{\bf Conflict of interest.} The authors declare that they have no conflict of interest.


\bibliographystyle{plain}

\end{document}